\documentclass[12pt]{article}

\usepackage{amsmath,amssymb,amsthm}
\usepackage[colorlinks=true,allcolors=black,bookmarksopen,bookmarksdepth=3]{hyperref}
\usepackage[margin=1in]{geometry}

\usepackage{enumitem}
\setitemize{label=$\bullet$,topsep=1pt,itemsep=1pt,partopsep=0pt,parsep=0pt}

\newtheorem{theorem}{Theorem}[section]
\newtheorem{lemma}[theorem]{Lemma}
\newtheorem{proposition}[theorem]{Proposition}
\newtheorem{corollary}[theorem]{Corollary}

\theoremstyle{definition}
\newtheorem{definition}[theorem]{Definition}

\theoremstyle{remark}
\newtheorem{remark}[theorem]{Remark}
\newtheorem{example}[theorem]{Example}

\newcommand{\QQ}{\mathbb Q}
\newcommand{\ZZ}{\mathbb Z}
\newcommand{\CC}{\mathbb C}
\newcommand{\RR}{\mathbb R}
\newcommand{\BB}{\mathbb B}
\newcommand{\EE}{\mathbb E}
\newcommand{\FF}{\mathbb F}
\newcommand{\PP}{\mathbb P}
\newcommand{\cP}{\mathcal P}

\newcommand{\im}{\operatorname{im}}
\newcommand{\id}{\operatorname{id}}
\newcommand{\tr}{\operatorname{tr}}
\newcommand{\ch}{\operatorname{ch}}
\newcommand{\st}{\operatorname{st}}
\newcommand{\Eq}{\operatorname{Eq}}
\newcommand{\Hom}{\operatorname{Hom}}
\newcommand{\Ext}{\operatorname{Ext}}
\newcommand{\End}{\operatorname{End}}
\newcommand{\Emb}{\operatorname{Emb}}

\newcommand{\supp}{\operatorname{supp}}
\newcommand{\cyl}{\operatorname{cyl}}
\newcommand{\Vect}{\operatorname{Vect}}
\newcommand{\Tri}{\operatorname{Tri}}
\newcommand{\Kom}{\operatorname{Kom}}

\newcommand{\Top}{\mathsf{Top}}
\newcommand{\Set}{\mathsf{Set}}
\newcommand{\Grpd}{\mathsf{Grpd}}
\newcommand{\Simp}{\mathsf{Simp}}
\newcommand{\s}{\mathsf s}
\newcommand{\C}{\mathsf C}
\newcommand{\D}{\mathsf D}
\newcommand{\op}{\mathsf{op}}
\newcommand{\Sm}{\mathsf{Sm}}

\newcommand{\even}{\mathrm{even}}

\newcommand{\sm}{\mathrm{sm}}
\newcommand{\per}{\mathrm{per}}
\newcommand{\kom}{\mathrm{kom}}

\newcommand{\acts}{\curvearrowright}
\newcommand{\righttwoarrows}{\mathrel{\vcenter{\mathsurround0pt{\ialign{##\crcr\noalign{\nointerlineskip}$\rightarrow$\crcr\noalign{\nointerlineskip}$\rightarrow$\crcr}}}}}
\newcommand{\rightthreearrows}{\mathrel{\vcenter{\mathsurround0pt{\ialign{##\crcr\noalign{\nointerlineskip}$\rightarrow$\crcr\noalign{\nointerlineskip}$\rightarrow$\crcr\noalign{\nointerlineskip}$\rightarrow$\crcr}}}}}

\begin{document}

\title{Enough vector bundles on orbispaces}

\author{John Pardon}

\date{10 December 2021}

\maketitle

\begin{abstract}
We show that every orbispace satisfying certain mild hypotheses has `enough' vector bundles.
It follows that the $K$-theory of finite rank vector bundles on such orbispaces is a cohomology theory.
Global presentation results for smooth orbifolds and derived smooth orbifolds also follow.
\end{abstract}

\section{Introduction}

A \emph{(separated) orbispace} \cite{haefligerorbiespaces,behrendintroduction} is a topological stack $X$ which admits a cover by open substacks of the form $Y/\Gamma$ (where $\Gamma\acts Y$ is a continuous action of a finite group on a topological space) and whose diagonal $X\to X\times X$ is proper (see \S\ref{stackssec} for background on topological stacks).
Familiar examples of orbispaces include orbifolds, graphs of groups, complexes of groups, and (the analytifications of) separated Deligne--Mumford algebraic stacks over $\CC$.

An interesting question to ask about a given orbispace $X$ is whether there exists a global presentation $X=Y/G$ for $G$ a compact Lie group; let us call such an orbispace a \emph{global quotient}.
There are a number of known conditions which imply an orbispace is a global quotient.
If $X=Y/\Gamma$ for a (possibly infinite) discrete group $\Gamma$, then $X$ is a global quotient by L\"uck--Oliver \cite[Corollary 2.7]{luckoliver}.
Every (paracompact) smooth \emph{effective} $n$-dimensional orbifold is a global quotient (of its orthonormal frame bundle by $O(n)$), and it is an old question whether every (not necessarily effective) smooth orbifold is a global quotient.
A sufficient criterion for an orbifold to be a global quotient was given by Henriques--Metzler \cite{henriquesmetzler}.
Henriques \cite{henriquesthesis} conjectured that every compact orbispace is a global quotient, however other experts have expressed skepticism that such a general result would be true \cite[\S 6.4]{siebertgw}.
The analogous question for algebraic stacks has been studied by Edidin--Hassett--Kresch--Vistoli \cite{ehkv} and Totaro \cite{totaro}.
It is a result of Kresch--Vistoli \cite[Theorem 2]{kreschvistoli} \cite[Theorem 4.4]{kresch} (using a result of Gabber \cite{dejonggabber} \cite[Chapter 4]{colliottheleneskorobogatov}) that smooth separated Deligne--Mumford stacks over $\CC$ with quasi-projective coarse moduli space are global quotients.

Our main result implies that all orbispaces satisfying very mild hypotheses are global quotients.
In particular, all compact orbispaces are global quotients.

\begin{theorem}\label{maingeneral}
Let $X$ be an orbispace, with isotropy groups of bounded order, whose coarse space $\left|X\right|$ is coarsely finite-dimensional (every open cover has a locally finite refinement with finite-dimensional nerve).
Then there exists a complex vector bundle $V$ of rank $n>0$ over $X$, whose fiber over every $x\in X$ is isomorphic to a direct sum of copies of the regular representation of the isotropy group $G_x$.
We may take $n=n(d,m)$ if $\left|X\right|$ is $d$-dimensional (every open cover has a locally finite refinement with nerve of dimension $\leq d$) and has isotropy groups of order $\leq m$.
\end{theorem}

(Note that the `real' and `complex' versions of Theorem \ref{maingeneral} are equivalent, by tensoring from $\RR$ to $\CC$ and by forgetting from $\CC$ to $\RR$.)

The proof of Theorem \ref{maingeneral} is split into two parts.
In \S\ref{mainsection}, we prove Theorem \ref{maingeneral} for any orbispace presented by a simplicial complex of groups; this special case carries the essential topological content of the result.
Due to one particular step in this proof, we have no explicit bound on the rank $n(d,m)$ of the vector bundles proven to exist.
In \S\ref{covernerve}, we deduce Theorem \ref{maingeneral} in general by showing that every paracompact orbispace admits a representable map to a simplicial complex of groups.
This result (Proposition \ref{maptosimplicial}) is likely of independent interest; it is the analogue of mapping a paracompact Hausdorff space to the nerve of an open covering using a partition of unity (it is thus worthy of note that its proof is not trivial).

The following immediate corollaries of Theorem \ref{maingeneral} are derived in \S\ref{vectorprincipal}.

\begin{corollary}\label{maingeneralprincipal}
For $X$ as in Theorem \ref{maingeneral}, we have $X=P/U(n)$ for a space $P$.
\end{corollary}

\begin{corollary}\label{orbifoldpresentation}
Every paracompact smooth orbifold $X$ of dimension $\leq d$ with isotropy groups of order $\leq m$ is the quotient $X=P/U(n)$ of a smooth manifold $P$ by a smooth action of the compact Lie group $U(n)$, where $n=n(d,m)$.
\end{corollary}

It seems that Theorem \ref{maingeneral} does not help resolve the question of whether every separated Deligne--Mumford stack of finite type over $\CC$ is a global quotient; the question of whether the vector bundles produced by Theorem \ref{maingeneral} on its analytification are algebraic (or even analytic) seems difficult.

In \S\ref{ktheory}, we prove the following corollary of Theorem \ref{maingeneral}, to which the titular phrase `having enough vector bundles' refers.

\begin{corollary}\label{enoughvectorbundles}
Let $X\to Y$ be a representable map of orbispaces satisfying the hypothesis of Theorem \ref{maingeneral}.
Every vector bundle on $X$ of bounded rank embeds into the pullback of a vector bundle of bounded rank on $Y$.
\end{corollary}

It is well known that Corollary \ref{enoughvectorbundles} is the key statement needed to show that the $K$-theory of finite rank vector bundles on orbispaces satisfies excision and exactness and is thus a cohomology theory (see \cite[\S 3]{luckoliver} and \cite[\S 6.3]{henriquesthesis}).
We elaborate on this assertion in \S\ref{ktheory}, using the suggestive reformulation of Corollary \ref{enoughvectorbundles} as the statement that pullback of vector bundles is \emph{cofinal}.
The $K$-theory of finite rank vector bundles should agree (for reasonable orbispaces) with the other standard models of $K$-theory for orbispaces, such as using bundles of Fredholm operators \cite{segal,matumoto,atiyahalgebraic,atiyahsinger} or using orthogonal spectra \cite[\S\S 6.3--6.4]{schwedeglobal} (compare Remark \ref{globalcomparison} below).

\begin{remark}
Another known (to experts) consequence of Corollary \ref{enoughvectorbundles} (which we do not explain in detail) is that every paracompact quasi-smooth derived smooth orbifold with tangent and obstruction spaces of dimension $\leq d$ and isotropy groups of order $\leq m$ is the derived zero set of a smooth section of a vector bundle of rank $\leq n$ over a smooth orbifold of dimension $\leq n=n(d,m)$.
(`Quasi-smooth' means locally isomorphic to the derived zero set of a smooth section of a smooth vector bundle over a smooth orbifold.)
\end{remark}

\begin{remark}\label{globalcomparison}
Combined work of Schwede \cite{schwedeequivalence} and Gepner--Henriques \cite{gepnerhenriques} establishes an equivalence between orthogonal spaces up to global equivalence (with respect to the `global family' of all finite groups) and certain categories of `cellular' topological stacks (which include, but are more general than, what we call orbispaces here).
The vector bundles produced by Theorem \ref{maingeneral} allow for a concrete description of the functor from orbispaces to orthogonal spaces (compare \cite[Definition 1.1.27]{schwedeglobal}).
Let $X$ be an orbispace, and let $E$ be any \emph{faithful} vector bundle over $X$ (meaning the fibers of $E$ are faithful representations of the isotropy groups of $X$).
The orthogonal space corresponding to $X$ is given by
\begin{equation*}
V\mapsto\Emb_X(E,\underline V)
\end{equation*}
where $\Emb_X(E,\underline V)$ denotes the total space of the bundle of embeddings of $E$ into $V$ (note that $\Emb_X(E,\underline V)$ is a space since $E$ is faithful).

Schwede \cite{schwedeglobal} also associates to every orthogonal spectrum a cohomology theory on orthogonal spaces, hence on orbispaces.
Given an orthogonal spectrum $A$ and an orbispace $X$ which admits faithful vector bundles, the degree zero $A$-cohomology of $X$ is (in view of the above) given by the direct limit over vector bundles $E$ over $X$ of the set of homotopy classes of sections of the fibration $\Omega^EA(E)\to X$.
More generally, we may consider the mapping spectrum $F(X,A)$ defined by
\begin{equation*}
n\mapsto\varinjlim_{E/X}\Gamma(X,\Omega^EA(E\oplus\RR^n))
\end{equation*}
whose stable homotopy groups are the $A$-cohomology groups of $X$.
If $A$ is a global $\Omega$-spectrum \cite[Definition 4.3.8]{schwedeglobal}, then this direct limit is achieved at any faithful $E$, and the above definition of $F(X,A)$ is an $\Omega$-spectrum.
Let us also propose a possible definition of the $A$-homology groups of $X$ as the stable homotopy groups of the spectrum
\begin{equation*}
n\mapsto\varinjlim_{E/X}\left|\Omega^EA(E\oplus\RR^n)\right|
\end{equation*}
where $\left|\cdot\right|$ indicates taking the coarse space of the total space of $\Omega^EA(E\oplus\RR^n)$ over $X$.
\end{remark}

\begin{remark}
It is natural to ask to what extent Theorem \ref{maingeneral} may be generalized to the case of `Lie orbispaces' (topological stacks locally modelled on $Y/G$ for $G$ a compact Lie group).
The naive generalization is simply false: there are purely ineffective Lie orbispaces with isotropy group $S^1$ and coarse space $S^3$ which have no finite rank faithful vector bundles \cite[\S 2]{totaro}.
It is, however, reasonable to conjecture that the proof of Theorem \ref{maingeneral} could be generalized to prove that the existence of enough vector bundles on a Lie orbispace is a purely cohomological question.
\end{remark}

\subsection{Acknowledgements}

I would like to thank Andr\'e Henriques for discussions about orbispaces and topological stacks, David Treumann for conversations about the algebraic case, and the referees for their remarks.
I also thank the Department of Mathematics at Cambridge University for their hospitality during the time when this paper was conceived.
This research was conducted during the period the author served as a Clay Research Fellow and was partially supported by a Packard Fellowship and by the National Science Foundation under the Alan T.\ Waterman Award, Grant No.\ 1747553.

\section{The main construction}\label{mainsection}

This section is devoted to proving Theorem \ref{maingeneral} for orbispaces which are presented by a simplicial complex of groups.
This special case (stated as Theorem \ref{mainsimplicial} below) carries the essential topological content of Theorem \ref{maingeneral}.
For general background on orbispaces and topological stacks, we refer the reader to \S\ref{stackssec}.

\medskip

We begin by describing the basic idea of the proof, which we then implement in detail.
Our orbispace $X$ comes with a filtration by skeleta
\begin{equation*}
\varnothing=X_{-1}\subseteq X_0\subseteq X_1\subseteq\cdots
\end{equation*}
where $X_k$ is obtained from $X_{k-1}$ by attaching cells of the form $(D^k,\partial D^k)\times\BB G$ for finite groups $G$ (here $\BB G=*/G$ is the orbispace quotient).
We will construct the desired vector bundle by induction on skeleta.

A direct implementation of this strategy runs immediately into the following obstruction.
A (complex) vector bundle $V$ on $\partial D^k\times\BB G$ decomposes canonically as $V=\bigoplus_{\rho\in\hat G}V_\rho\otimes\rho$ for vector bundles $V_\rho$ on $\partial D^k$ indexed by the complex irreducible representations $\rho$ of $G$.
Thus $V$ is classified by an element of $\prod_{\rho\in\hat G}\pi_{k-1}(\coprod_nBU(n))$, which vanishes iff $V$ extends to $D^k\times\BB G$.
We shall seek to detect these obstructions using Chern characters.
According to Bott periodicity, the homotopy groups of $BU$ are given by
\begin{equation*}
\pi_i(BU)=\begin{cases}\ZZ&i\text{ even,}\\0&i\text{ odd,}\end{cases}
\end{equation*}
and the Chern character of a generator of $\pi_{2i}(BU)$ is nonzero.
Now the Chern character of $V=\bigoplus_{\rho\in\hat G}V_\rho\otimes\rho$ is given by $\ch(V)=\sum_{\rho\in\hat G}\ch(V_\rho)\dim\rho$, so triviality of $\ch(V)$ does not imply triviality of each $\ch(V_\rho)$ or of the aforementioned obstructions to extending $V$.

To capture the information we need about $V$, we consider the more refined characteristic class which we will call the \emph{inertial Chern character}\footnote{It is tempting to call it the ``Chern character character''.} $\ch^I$ (studied previously by Adem--Ruan \cite{ademruan}), which is a cohomology class on the \emph{inertia stack}
\begin{equation*}
IX:=\Eq\bigl(X\righttwoarrows X\bigr)=X\times_{X\times X}X.
\end{equation*}
In local coordinates $X=Y/G$, we have $IX=\bigl(\bigsqcup_{g\in G}Y^g\bigr)/G$.
In such coordinates, the inertial Chern character is defined by recalling the Chern--Weil description of the usual Chern character $\ch(V)=\tr\exp(\Omega/(-2\pi i))$, for $\Omega$ the curvature of a hermitian connection on $V$, and writing
\begin{equation*}
\ch^I(V)=\tr\Bigl[g\exp\Bigl(\frac\Omega{-2\pi i}\Bigr)\Bigr].
\end{equation*}
Now the inertia stack of $\partial D^k\times\BB G$ is $\partial D^k\times G/G$ (quotient by the conjugation action), and the inertial Chern character of $V=\bigoplus_{\rho\in\hat G}V_\rho\otimes\rho$ is given by $\ch^I(V)=\sum_{\rho\in\hat G}\ch(V_\rho)\chi_\rho$.
Thus if $\ch^I(V)$ is trivial, then so is each $\ch(V_\rho)$ by linear independence of characters, and hence $V$ extends to $D^k\times\BB G$ as desired.

We are thus led to consider the modified problem of constructing a vector bundle on $X$ with the desired fibers and with trivial inertial Chern character.
It suffices to show that if $X_{k-1}$ admits such a vector bundle, then so does $X_k$.
The vanishing of $\ch^I(V)\in H^*(IX_{k-1})$ guarantees that the obstructions to extending $V$ to $X_k$ vanish, by the above discussion.
It thus remains to show that this extension can be taken to have trivial inertial Chern character.
The inertial Chern character of any extension is an element of $\ker(H^*(IX_k)\to H^*(IX_{k-1}))$.
By the long exact sequence
\begin{equation*}
\cdots\to H^*(IX_k,IX_{k-1})\to H^*(IX_k)\to H^*(IX_{k-1})\to\cdots
\end{equation*}
it is thus in the image of $H^*(IX_k,IX_{k-1})$.
The inertial Chern character is an even degree class, so we may assume that $k$ is even.
By modifying how we extend $V$ from $X_{k-1}$ to $X_k$ by an element of $\prod_{\rho\in\hat G}\pi_k(BU)$ over a given $k$-simplex, we can shift its inertial Chern character by any integral linear combination of characters in $\Hom(G/G,\CC)\subseteq H^*(IX_k,IX_{k-1})$ (for $G$ the isotropy group of that given $k$-simplex).
By replacing $V$ with $V^{\oplus a}$, we can multiply its inertial Chern character by any positive integer $a$.
A combination of these two operations suffice to kill the inertial Chern character, provided that it is \emph{$\chi$-rational} (rational with respect to a certain $\QQ$-structure on $H^*(IX)$ differing from the usual one).
This rationality is not at all obvious given the transcendental definition we give of the inertial Chern character, but it is true, and thus the proof is completed.

\medskip

The remainder of this section is devoted to making the above outline precise.
We begin by recalling the definition of a simplicial complex of groups (see also Haefliger \cite{haefliger}, Corson \cite{corson}, or Bridson--Haefliger \cite{bridsonhaefliger}).

\begin{definition}
A \emph{simplicial complex of groups} is a pair $(Z,G)$ consisting of a simplicial complex $Z$ together with the following data:
\begin{itemize}
\item For every simplex $\sigma\subseteq Z$, a group $G_\sigma$.
\item For every pair of simplices $\sigma\subseteq\tau$, an \emph{injective} group homomorphism $G_\tau\hookrightarrow G_\sigma$.
\item For every triple of simplices $\rho\subseteq\sigma\subseteq\tau$, an element of $G_\rho$ conjugating the inclusion $G_\tau\hookrightarrow G_\rho$ to the composition of inclusions $G_\tau\hookrightarrow G_\sigma\hookrightarrow G_\rho$.
\item For every quadruple of simplices $\pi\subseteq\rho\subseteq\sigma\subseteq\tau$, the resulting product of elements of $G_\pi$ conjugating $G_\tau\hookrightarrow G_\pi$ to $G_\tau\hookrightarrow G_\rho\hookrightarrow G_\pi$ to $G_\tau\hookrightarrow G_\sigma\hookrightarrow G_\rho\hookrightarrow G_\pi$ to $G_\tau\hookrightarrow G_\sigma\hookrightarrow G_\pi$ and back to $G_\tau\hookrightarrow G_\pi$ must be the identity element of $G_\pi$.
\end{itemize}
Injectivity of each map $G_\tau\hookrightarrow G_\sigma$ will ensure that the geometric realization of $(Z,G)$ is an orbispace, rather than some sort of more exotic topological stack.
\end{definition}

A simplicial complex of groups $(Z,G)$ presents an orbispace $\left\|(Z,G)\right\|$ called its \emph{geometric realization}.
A precise definition of this geometric realization is given in \S\ref{covernerve}.
For now, it will suffice to know that the coarse space of $\left\|(Z,G)\right\|$ is (the geometric realization of) $Z$ itself, and that over the open star $\st(\sigma)\subseteq Z$ of a simplex $\sigma\subseteq Z$, the geometric realization $\left\|(Z,G)\right\|$ is given by the orbispace quotient
\begin{equation*}
\biggl(\bigsqcup_{\tau\supseteq\sigma}\tau\times(G_\tau\backslash G_\sigma)\biggr)\biggm/G_\sigma
\end{equation*}
where the pieces for $\tau\supseteq\tau'\supseteq\sigma$ are glued together via the map $G_\tau\backslash G_\sigma\to G_{\tau'}\backslash G_\sigma$ induced by the map $G_\tau\to G_{\tau'}$ and the element of $G_\sigma$ conjugating the composition $G_\tau\to G_{\tau'}\to G_\sigma$ to the map $G_\tau\to G_\sigma$.

We will often abuse terminology and say `simplicial complex of groups' when we really mean its geometric realization.
Thus we will refer to $X=\left\|(Z,G)\right\|$ as a simplicial complex of groups.
Its coarse space $\left|X\right|$ is the geometric realization of $Z$.

To apply the methods of differential topology to a given simplicial complex of groups, we fix a family of (germs of) smooth retractions $\tau\to\sigma$ for every pair of simplices $\sigma\subseteq\tau$ such that the maps $\tau\to\sigma\to\rho$ and $\tau\to\rho$ agree for $\rho\subseteq\sigma\subseteq\tau$ (such a family of smooth retractions may be constructed by induction).
Given this data, objects of differential topology on $\tau$ (functions, differential forms, bundles, connections, etc.)\ are required to be pulled back under $\tau\to\sigma$ in a(n unspecified) neighborhood of every $\sigma\subseteq\tau$.
For bundles, this requirement consists of the \emph{data} of a compatible family of isomorphisms with the pullback bundles (one could equivalently consider only vector bundles built out of transition functions which satisfy the given pullback conditions).

In particular, the de Rham complex $\Omega^*(X)$ of a simplicial complex of groups $X$ is defined, and it coincides with the de Rham complex of its coarse space $\Omega^*(\left|X\right|)$.
There is a natural map $\Omega^*(\left|X\right|)\to C^*(\left|X\right|;\RR)=C^*(X;\RR)$ from the de Rham complex to the simplicial cochain complex over $\RR$, given by integrating differential forms over oriented simplices.
It is a standard fact that this map is a quasi-isomorphism (proof: filter $X$ by skeleta and invoke the five-lemma to reduce to showing that $\Omega^*(\Delta^k,\partial\Delta^k)\to C^*(\Delta^k,\partial\Delta^k;\RR)$ is a quasi-isomorphism, which follows from the Poincar\'e lemma).

\begin{definition}[Chern character]
Let $X$ be a simplicial complex of groups, and let $V\to X$ be a complex vector bundle.
Given a hermitian metric $\mu$ and hermitian connection $\theta$ on $V$, the curvature of $\theta$ is a $2$-form $\Omega(V,\theta)$ valued in $\mathfrak u(V,\mu)\subseteq\mathfrak{gl}(V)=\End(V)$, and the Chern character form
\begin{equation*}
\ch(V,\theta):=\tr\exp\Bigl(\frac{\Omega(V,\theta)}{-2\pi i}\Bigr)\in\Omega^\even(X)
\end{equation*}
is closed.
Its class in cohomology $\ch(V)\in H^\even(X;\RR)$ is called the Chern character of $V$.
This class is independent of $\mu$ and $\theta$ by an interpolation argument (interpolate on $X\times[0,1]$ between any two metric/connection pairs on $X\times 0$ and $X\times 1$, and use the fact that $H^*(X\times[0,1])=H^*(X)$).
\end{definition}

For any stack $X$, its \emph{inertia stack} is the stack
\begin{equation*}
IX:=\Eq\bigl(X\righttwoarrows X\bigr)=X\times_{X\times X}X.
\end{equation*}
When $X=Y/G$, we have $IX=\bigl(\bigsqcup_{g\in G}Y^g\bigr)/G$.
In particular, the local description of the geometric realization of a simplicial complex of groups provides a local description of its inertia stack as well.
In fact, if $X$ is a simplicial complex of groups then so is $IX$: a simplex of $IX$ is a pair $(\sigma,[g])$ where $\sigma\subseteq X$ is a simplex and $[g]\subseteq G_\sigma$ is a conjugacy class, and $G_{(\sigma,[g])}$ is the centralizer of any element of the conjugacy class $[g]\subseteq G_\sigma$, etc.

\begin{definition}[Inertial Chern character]
Let $X$ be a simplicial complex of groups, and let $V\to X$ be a complex vector bundle.
The inertial Chern character $\ch^I(V)\in H^\even(IX;\CC)$ is represented by the closed form
\begin{equation*}
\ch^I(V,\theta):=\tr\Bigl[g\exp\Bigl(\frac\Omega{-2\pi i}\Bigr)\Bigr]\in\Omega^\even(IX;\CC)
\end{equation*}
for any choice of hermitian metric and connection $\theta$ on $V$.
Closedness can be seen by splitting the pullback of $V$ to $IX$ into isotypic pieces for the action of the cyclic group generated by $g$ and observing that the contribution of each such piece is a (usual) Chern character form.
Independence of the metric and connection follows from the same interpolation argument as before.
\end{definition}

\begin{example}
If $X=\BB G$, a vector bundle over $X$ is simply a representation of $G$, the inertia stack $IX=G/G$ is the stack quotient of the conjugation action of $G$ on itself, and the inertial Chern character $\ch^I(V):G/G\to\CC$ is the character $g\mapsto\tr(g|V)$ of $V$ regarded as a representation of $G$.
\end{example}

\begin{example}
If $X=Y\times\BB G$ for a smooth manifold $Y$ and $V=\bigoplus_{\rho\in\hat G}V_\rho\otimes\rho$, then $IX=Y\times(G/G)$ and the inertial Chern character $\ch^I(V):G/G\to H^*(Y;\CC)$ is given by $\ch^I(V)=\sum_\rho\ch(V_\rho)\chi_\rho$.
Since the characters $\chi_\rho$ of the irreducible representations $\rho$ of $G$ form a basis for the space of maps $G/G\to\CC$, we see that for $X=Y\times\BB G$, the inertial Chern character determines (and is determined by) the Chern characters of each of the associated bundles $V_\rho$.
\end{example}

\begin{example}
The degree zero part of the inertial Chern character $\ch_0^I(V)\in H^0(IX;\CC)$ records the characters of the fibers of $V$, regarded as representations of the isotropy groups $G_x$ of the points of $X$.
\end{example}

Despite admitting the aforementioned transcendental definition in terms of differential forms, the (usual) Chern character $\ch(V)\in H^*(X;\RR)$ is well known be rational, i.e.\ it lies in the subspace $H^*(X;\QQ)\subseteq H^*(X;\RR)$.
The inertial Chern character does not always lie in the subspace $H^*(IX;\QQ)\subseteq H^*(IX;\CC)$, rather it is rational with respect to a different $\QQ$-structure, which we now introduce.

\begin{definition}[$\chi$-integral cohomology of $IX$]\label{chiratdef}
Regard the simplicial cochain group $C^*(IX;\CC)$ as the group of simplicial cochains on $X$ with coefficients in $\Hom(G_x/G_x,\CC)$.
In this description, let us replace $\Hom(G_x/G_x,\CC)$ with its subspace of integral (resp.\ rational) linear combinations of characters of $G_x$.
Equivalently, we consider the complex of simplicial cochains on $IX$ which act on $\sigma\times(G_\sigma/G_\sigma)\subseteq IX$ (for any simplex $\sigma\subseteq X$) by an integral (resp.\ rational) linear combination of characters of $G_\sigma$.
We denote the resulting cohomology groups by $H^*(IX;\ZZ_\chi)$ (resp.\ $H^*(IX;\QQ_\chi)$), which are evidently functorial in $X$.
We say that an element of $H^*(IX;\CC)$ is \emph{$\chi$-integral} (resp.\ \emph{$\chi$-rational}) to mean that it lies in the image of $H^*(IX;\ZZ_\chi)$ (resp.\ in the subspace $H^*(IX;\QQ_\chi)$).
\end{definition}

\begin{example}
If $X=Y\times\BB G$ then $IX=Y\times(G/G)$ and $H^*(IX;\ZZ_\chi)=H^*(X;\ZZ[\hat G])$.
Thus $\ch^I(V)=\sum_\rho\ch(V_\rho)\chi_\rho$ is $\chi$-rational since each $\ch(V_\rho)$ is rational.
\end{example}

Since the coefficient systems on $X$ appearing in Definition \ref{chiratdef} are finite free modules over $\ZZ$, we may dualize to define homology groups $H_*(IX;\ZZ_\chi)$ as well.

\begin{lemma}\label{uct}
There are canonical isomorphisms
\begin{align*}
H^*(IX;\CC)&=\Hom(H_*(IX;\ZZ_\chi),\CC),\\
H^*(IX;\QQ_\chi)&=\Hom(H_*(IX;\ZZ_\chi),\QQ),
\end{align*}
and a canonical short exact sequence
\begin{equation*}
0\to\Ext^1(H_{*-1}(IX;\ZZ_\chi),\ZZ)\to H^*(IX;\ZZ_\chi)\to\Hom(H_*(IX;\ZZ_\chi),\ZZ)\to 0.
\end{equation*}
\end{lemma}

\begin{proof}
The complex $C_*(IX;\ZZ_\chi)$ is degreewise free, and the complexes $C^*(IX;\ZZ_\chi)$, $C^*(IX;\QQ_\chi)$, and $C^*(IX;\CC)$ are obtained from it by applying the functors $\Hom(-,\ZZ)$, $\Hom(-,\QQ)$, and $\Hom(-,\CC)$, respectively.
\end{proof}

\begin{proposition}[$\chi$-rationality of $\ch^I$]\label{inertialchernrational}
Let $X$ be a simplicial complex of groups of order $\leq m$.
The inertial Chern character of any vector bundle $V/X$ is $\chi$-rational.
In fact, there exists a positive integer $N=N(m,d,\dim V)$ such that $N\cdot\ch^I_d(V)$ is $\chi$-integral.
\end{proposition}

It is worth remarking that the proof we give does not provide an explicit expression for $N(m,d,n)$ other than for $n=1$.

\begin{proof}
We first show that the inertial Chern character of any line bundle $L/X$ is $\chi$-rational.
We have $\tr(ab)=\tr(a)\tr(b)$ for endomorphisms $a,b$ of $\CC$, so the inertial Chern character of a line bundle $L$ splits as the product
\begin{equation*}
\ch^I(L)=\ch_0^I(L)\exp(\ch_1(L))
\end{equation*}
where $\ch_0^I(L)\in H^0(IX;\ZZ_\chi)$ is the fiberwise character of $L$ and $\ch_1(L)\in H^2(X;\CC)$ is the Chern character in degree one.
It thus suffices to show that $\ch_1(L)$ is rational.
We have $k\ch_1(L)=\ch_1(L^{\otimes k})$, so it suffices to show that $\ch_1(L^{\otimes k})$ is integral for some positive integer $k$.
Since all isotropy groups of $X$ have order $\leq m$, the tensor power $L^{\otimes m!}$ descends to a line bundle $M$ on the coarse space.
Since $M$ is an ordinary line bundle over a space (rather than an orbispace), its first Chern class $c_1(M)=\ch_1(M)$ is integral.
We have thus shown that $\ch^I(L)$ is $\chi$-rational (in fact, we have shown that $d!(m!)^d\ch^I_d(L)$ is $\chi$-integral).

To treat the case of vector bundles of dimension greater than one, we will use the \emph{splitting principle}.
Given a vector bundle $V/X$ with hermitian metric, let $\FF(V)\to X$ denote the fibration whose fiber over $x\in X$ is the space of decompositions of $V_x$ into ordered orthogonal one-dimensional subspaces.
In other words, $\FF(V)=P/U(1)^n$ where $P\to X$ is the principal $U(n)$-bundle associated to $V/X$ with its chosen metric.
We claim that the pullback map
\begin{equation*}
H^*(IX;\CC)\to H^*(I\FF(V);\CC)
\end{equation*}
is injective.
The fiber of $I\FF(V)\to IX$ over a given point $(x,g)\in IX$ is the space of $g$-invariant ordered decompositions of $V_x$ into one-dimensional subspaces.
Since the group generated by $g$ is abelian, its irreducible representations are one-dimensional, and hence there are plenty of such decompositions which are $g$-invariant, namely when each one-dimensional subspace is contained in some $g$-isotypic piece of $V_x$.
Thus $I\FF(V)\to IX$ is a disjoint union of iterated projective space bundles, from which the desired injectivity statement follows (for any projective space bundle $\PP(W)\to Z$ over an orbispace $Z$, the pullback map $H^*(Z)\to H^*(\PP(W))$ is split by the map $\alpha\mapsto\int(c_1(L)^{\dim W-1}\cup\alpha)$ where $L/\PP(W)$ denotes the tautological line bundle and $\int:H^{*+2(\dim W-1)}(\PP(W))\to H^*(Z)$ denotes fiberwise integration).

Since $H^*(IX;\CC)\to H^*(I\FF(V);\CC)$ is injective and $\Hom_\ZZ(-,\CC)$ is exact, we conclude from Lemma \ref{uct} that $\Hom(H_*(IX;\ZZ_\chi)/H_*(I\FF(V);\ZZ_\chi),\CC)=0$, and hence that the domain is torsion.
From this and Lemma \ref{uct} again, it follows that a class in $H^*(IX;\CC)$ is $\chi$-rational iff its pullback to $H^*(I\FF(V);\CC)$ is $\chi$-rational.
Now consider the pullback of $V$ to $\FF(V)$.
This pullback splits as a direct sum of line bundles, so since the inertial Chern character is additive under direct sum, we conclude that the pullback of $\ch^I(V)$ to $I\FF(V)$ is $\chi$-rational, and hence that $\ch^I(V)$ is itself $\chi$-rational.

It remains to produce an integer $N=N(m,d,n)$ such that $N\cdot\ch^I_d(V)$ is $\chi$-integral for $\dim V=n$.
We claim that there exists a \emph{finite} orbi-complex $\BB_{m,d}U(n)$ carrying a principal $U(n)$-bundle
\begin{equation*}
\EE_{m,d}U(n)\to\BB_{m,d}U(n)
\end{equation*}
with the property that every principal $U(n)$-bundle over a simplicial complex of groups $X$ of dimension $\leq d$ and isotropy $\leq m$ is a pullback of $\EE_{m,d}U(n)\to\BB_{m,d}U(n)$.
Since the inertial Chern character of $\EE_{m,d}U(n)\to\BB_{m,d}U(n)$ is $\chi$-rational and $\BB_{m,d}U(n)$ is finite, there exists an integer $N=N(m,d,n)$ such that $N$ times this inertial Chern character is $\chi$-integral in cohomological degree $d$.
By pullback, the same integer $N$ works for any vector bundle of rank $n$ over a simplicial complex of groups of order $\leq m$.
Note that this argument gives no explicit bound on the integer $N$.

It remains to construct $\EE_{m,d}U(n)\to\BB_{m,d}U(n)$.
We seek $U(n)$-spaces $\EE_{m,d}U(n)$ such that
\begin{align*}
\EE_{m,-1}U(n)&=\varnothing,\\
\EE_{m,d}U(n)&=\EE_{m,d-1}U(n)\cup\biggl(\,\begin{matrix}\text{finitely many cells of the}\hfill\\\text{form }(D^d,\partial D^d)\times U(n)/G\hfill\end{matrix}\,\biggr),\\
\pi_r((\EE_{m,d}U(n))^G)&=0\quad\text{for }r<d,
\end{align*}
where $G\subseteq U(n)$ ranges over all finite subgroups of $U(n)$ of order $\leq m$.
Note that there are finitely many conjugacy classes of such subgroups, and it suffices to consider just one representative of each conjugacy class.
To show that $\EE_{m,d}U(n)$ exists given $\EE_{m,d-1}U(n)$, argue as follows.
We have $\pi_{d-1}((\EE_{m,d-1}U(n))^G)=H_{d-1}((\EE_{m,d-1}U(n))^G)$ by Hurewicz, and the latter group is finitely generated since $\EE_{m,d-1}U(n)$ is made up of finitely many cells.
We define $\EE_{m,d}U(n)$ by attaching cells $(D^d,\partial D^d)\times U(n)/G$ along a choice of finitely many generators of $\pi_{d-1}((\EE_{m,d-1}U(n))^G)$, for each of the finitely many subgroups $G\subseteq U(n)$ on our fixed set of representatives.
Note that $(D^d\times U(n)/H)^G=D^d\times\{a\in U(n)\,|\,a^{-1}Ga\subseteq H\}$, so any map $S^r\to(\EE_{m,d}U(n))^G$ for $r<d$ can be homotoped to land inside $(\EE_{m,d-1}U(n))^G$.
It follows that the homotopy groups of $(\EE_{m,d}U(n))^G$ satisfy the desired vanishing property.

Let us show that any principal $U(n)$-bundle $P\to X$ is pulled back from $\EE_{m,d}U(n)\to\BB_{m,d}U(n)$ when $\dim X\leq d$ with isotropy groups of order $\leq m$.
That is, we should construct an equivariant map $P\to\EE_{m,d}U(n)$ (note that since the target is a space, this is the same as an equivariant map from the coarse space $\left|P\right|$).
By induction on the cells of $X$, it suffices to solve the extension problem for equivariant maps from $(D^r,\partial D^r)\times U(n)/G$ to $\EE_{m,d}U(n)$ for $r\leq d$ and $\left|G\right|\leq m$.
This is equivalent to the extension problem for maps $(D^r,\partial D^r)\to(\EE_{m,d}U(n))^G$, whose positive solution for $r\leq d$ is one of the defining properties of $\EE_{m,d}U(n)$.
\end{proof}

\begin{example}
A previous version of this text claimed a version of Proposition \ref{inertialchernrational} with $N$ independent of $m$.
This stronger result is false; here is a counterexample.
Begin with $\BB(\ZZ/m)$ with a line bundle given by multiplication by $e^{2\pi i/m}$ on $\CC$.
Glue on a disk $D^2$ using an attaching map $\partial D^2=S^1\to\BB(\ZZ/m)$ to a generator of $\ZZ/m$.
The line bundle extends since all complex line bundles on a circle are trivial.
The inertial Chern character of this line bundle has denominator $m$.
\end{example}

\begin{lemma}[Bott]\label{cherncharacterimage}
The image of the Chern character map
\begin{equation*}
\ch_n:\pi_{2n}(BU)\to H^{2n}(S^{2n};\QQ)
\end{equation*}
is precisely $H^{2n}(S^{2n};\ZZ)$.
\end{lemma}

\begin{proof}
This is an immediate corollary of Bott periodicity; for completeness, we include the proof from \cite{mathewbott}.
Let $\eta/S^2$ denote the line bundle with $c_1(\eta)=1$, and let $1$ denote the trivial line bundle.
According to Bott periodicity, multiplication with $\eta-1\in\tilde K^0(S^2)$ defines an isomorphism $\tilde K^0(X)\to\tilde K^0(\Sigma^2X)$.
In particular, $(\eta-1)^{\otimes n}\in\tilde K^0(S^{2n})$ is a free generator.
Multiplicativity of the Chern character gives $\ch_n((\eta-1)^{\otimes n})=\ch_1(\eta-1)^n=1^n=1$.
\end{proof}

We now have all the ingredients we need to prove the main result of this section, namely Theorem \ref{maingeneral} for simplicial complexes of groups.

\begin{theorem}\label{mainsimplicial}
Let $X$ be a $d$-dimensional simplicial complex of groups of order $\leq m$.
There exists a complex vector bundle $V$ of rank $n=n(d,m)>0$ over $X$, whose fiber over $x\in X$ is isomorphic to a direct sum of copies of the regular representation of $G_x$.
\end{theorem}

\begin{proof}
The condition on the fibers of $V$ amounts to the assertion that $\ch^I_0(V)=n\chi_1$ where $\chi_1\in H^0(IX;\QQ_\chi)$ denotes ``the characteristic function of the identity'' as a function $G_x/G_x\to\CC$.
We will construct $V$ satisfying
\begin{equation*}
\ch^I(V)=n\chi_1
\end{equation*}
(meaning all higher inertial Chern characters vanish).
Certainly such a vector bundle exists over the $0$-skeleton $X_0$ of $X$, with rank $m!$, since all isotropy groups have order $\leq m$.
It therefore suffices to show that any such vector bundle on $X_{k-1}$ extends to $X_k$.

To extend $V$ as a vector bundle from $X_{k-1}$ to $X_k$ amounts to doing an extension from $\partial D^k\times\BB G$ to $D^k\times\BB G$ for each $k$-simplex.
The obstruction to doing this lies in $\prod_{\rho\in\hat G}\pi_{k-1}(BU(n_\rho))$, where $n_\rho=(\dim V)(\dim\rho)/\left|G\right|$.
The map $BU(n_\rho)\to BU$ is an isomorphism on $\pi_{k-1}$ provided $k-1<2n_\rho+2$, which is guaranteed by taking say $\dim V>dm$.
The Chern character detects $\pi_*(BU)$ by Lemma \ref{cherncharacterimage}, so vanishing of the inertial Chern character and linear independence of characters means that these obstructions all vanish.
Thus $V$ extends to $X_k$.

It remains to show that the extension of $V$ to $X_k$ can be taken to have trivial inertial Chern character.
The space of extensions over a given $k$-simplex is a torsor for $\prod_{\rho\in\hat G}\pi_k(BU(n_\rho))$, which is again $\prod_{\rho\in\hat G}\pi_k(BU)$ once we take $\dim V>dm$.
We will use this freedom to ensure that the inertial Chern character of $V$ on $X_k$ vanishes.
Consider the long exact sequence
\begin{equation*}
\cdots\to H^*(IX_k,IX_{k-1};\QQ_\chi)\to H^*(IX_k;\QQ_\chi)\to H^*(IX_{k-1};\QQ_\chi)\to\cdots.
\end{equation*}
Since the inertial Chern character in $H^*(IX_k;\QQ_\chi)$ maps to zero in $H^*(IX_{k-1};\QQ_\chi)$, it is in the image of $H^*(IX_k,IX_{k-1};\QQ_\chi)$.
In particular, it may be nonzero only in degree $k$.
Let us now replace $V$ with $V^{\oplus a}$ (which multiplies its inertial Chern character by $a$) for an appropriate positive integer $a$ so that by Proposition \ref{inertialchernrational} its inertial Chern character is $\chi$-integral.
Modifying our bundle by an element of (the product over the $k$-simplices of) $\prod_{\rho\in\hat G}\pi_k(BU)$ shifts its inertial Chern character by anything in $H^k(IX_k,IX_{k-1};\ZZ_\chi)$ by Lemma \ref{cherncharacterimage}.
We are done since $H^k(IX_k,IX_{k-1};\ZZ_\chi)\to H^k(IX_k;\ZZ_\chi)$ is surjective by the long exact sequence.
\end{proof}

\begin{remark}
One may interpret the proof of Theorem \ref{mainsimplicial} in homotopy theoretic terms as follows.
Vector bundles are classified by maps to a classifying space $\BB U(n)$, and since $\BB U(n)$ is not contractible, the extension problem for vector bundles has nontrivial obstructions.
Vector bundles with rationally trivialized inertial Chern character are classified by maps to the total space of a fibration over $\BB U(n)$ whose fiber classifies odd-dimensional rational cohomology classes on the inertia stack.
Our observation that the obstructions to this new problem are torsion is essentially the observation that this total space is rationally contractible.
The definition of this fibration over $\BB U(n)$ depends on the additivity of the Chern character (note that $\EE U(n)\to\BB U(n)$ is not suitable for this argument since $\EE U(n)$ is a space, hence every bundle pulled back from $\EE U(n)$ has trivial isotropy representations).
It may prove interesting to intepret this argument within Schwede's framework of global homotopy theory \cite{schwedeglobal}.
\end{remark}

\section{Topological stacks}\label{stackssec}

We review some basic facts about stacks (on the category of topological spaces), we give a precise definition of what we mean by an `orbispace', and we establish some of their basic properties.
For further background, the reader may wish to consult Noohi \cite{noohifoundations}, Gepner--Henriques \cite{gepnerhenriques}, Behrend \cite{behrendintroduction}, Metzler \cite{metzler}, Behrend--Noohi \cite{behrendnoohi}, Heinloth \cite{heinlothnotes}, or Laumon--Moret-Bailly \cite{lmb}.

Let $\Top$ denote the category of topological spaces and continuous maps, and let $\Grpd$ denote the 2-category of (essentially) small groupoids.
A \emph{stack} is a functor $F:\Top^\op\to\Grpd$ which satisfies \emph{descent}, i.e.\ such that for every topological space $U$ and every open cover $\{U_i\to U\}_i$, the natural functor
\begin{equation*}
F(U)\to\Eq\Bigl[\prod_iF(U_i)\righttwoarrows\prod_{i,j}F(U_i\cap U_j)\rightthreearrows\prod_{i,j,k}F(U_i\cap U_j\cap U_k)\Bigr]
\end{equation*}
is an equivalence.
Stacks form a 2-category, with morphisms given by natural transformations of functors.
The 2-category of stacks is complete, meaning all (small) limits exist; furthermore these limits may be calculated pointwise in the sense that $(\lim_\alpha F_\alpha)(U)=\lim_\alpha(F_\alpha(U))$.
Note that, as we are working in a 2-categorical context, all functors are 2-functors, all diagrams are 2-diagrams, all limits are 2-limits, etc.\ (though we will usually omit the prefix `2-').

The Yoneda lemma implies that the Yoneda functor $X\mapsto\Hom(-,X)$ embeds the category of topological spaces fully faithfully into the 2-category of stacks, and moreover that the natural map from $F(X)$ to the groupoid of maps of stacks $\Hom(-,X)\to F(-)$ is an equivalence.
The category of topological spaces is complete, and the Yoneda embedding is continuous (commutes with limits).
Hence we will make no distinction between a topological space $X$ and the associated stack $\Hom(-,X)$ of maps to $X$, nor between objects of $F(X)$ and maps $\Hom(-,X)\to F(-)$ (which we will simply write as $X\to F$).

Every stack $X$ has a \emph{coarse space} $\left|X\right|$ (a topological space) which is initial in the category of maps from $X$ to topological spaces.
Concretely, the points of $\left|X\right|$ are the isomorphism classes of maps $*\to X$, and a subset $U\subseteq\left|X\right|$ is open iff for every map $Y\to X$ from a topological space $Y$, the inverse image of $U$ is an open subset of $Y$.

A stack is called \emph{representable} iff it is in the essential image of the Yoneda embedding (i.e.\ it is isomorphic to a topological space).
A morphism of stacks $F\to G$ is called representable iff for every map $X\to G$ from a topological space $X$, the fiber product $F\times_GX$ is representable.

For any property $\cP$ of morphisms of topological spaces which is preserved under pullback, a representable morphism of stacks $F\to G$ is said to have property $\cP$ iff the pullback $F\times_GX\to X$ has $\cP$ for every map $X\to G$ from a topological space $X$.
The following are examples of properties $\cP$ of morphisms $f:X\to Y$ which are preserved under pullback:
\begin{itemize}
\item$f$ is \emph{injective}.
\item$f$ is \emph{surjective}.
\item$f$ is \emph{open}, meaning that the image of any open set is open.  In contrast, the property of being \emph{closed}, meaning that the image of any closed set is closed, is \emph{not} preserved under pullback.
\item$f$ is an \emph{embedding}, meaning that it is a homeomorphism onto its image.
\item$f$ is a \emph{closed embedding}.
\item$f$ is \emph{\'etale}, meaning that for every $x\in X$ there exists an open neighborhood $x\in U\subseteq X$ such that $f|_U:U\to Y$ is an open embedding.
\item$f$ is \emph{separated}, meaning that for every distinct pair $x_1,x_2\in X$ with $f(x_1)=f(x_2)$, there exist open neighborhoods $x_i\in U_i\subseteq X$ which are disjoint $U_1\cap U_2=\varnothing$.
(This is equivalent to the relative diagonal $X\to X\times_YX$ being a closed embedding.)
\item$f$ is \emph{universally closed}, meaning that $X\times_YZ\to Z$ is closed for every $Z\to Y$.
This is equivalent to the assertion that for every $y\in Y$ and every collection of open sets $\{U_i\subseteq X\}_i$ covering $f^{-1}(y)$, there exists a finite subcollection which covers $f^{-1}(V)$ for some open neighborhood $y\in V\subseteq Y$.
(One proof of this equivalence goes via yet a third equivalent condition, namely that every net $\{x_\alpha\in X\}_\alpha$ with $f(x_\alpha)\to y$ has a subnet converging to some $x\in f^{-1}(y)$.)
\item$f$ is \emph{proper}, meaning that it is separated and universally closed.
\item$f$ \emph{admits local sections}, meaning that there is an open cover $\{U_i\subseteq Y\}_i$ such that every restriction $f|_{f^{-1}(U_i)}:f^{-1}(U_i)\to U_i$ admits a section.
\item$f$ is a \emph{finite covering space}, meaning that there is an open cover $\{U_i\subseteq Y\}_i$ such that every restriction $f|_{f^{-1}(U_i)}:f^{-1}(U_i)\to U_i$ is isomorphic to $U_i^{\sqcup n_i}\to U_i$ for some integer $n_i\geq 0$.
\end{itemize}
Each of these properties $\cP$ is also closed under composition, and thus also under fiber products, meaning that for maps $X\to X'$ and $Y\to Y'$ over $Z$, if both $X\to X'$ and $Y\to Y'$ have $\cP$ then $X\times_ZY\to X'\times_ZY'$ also has $\cP$ (indeed, $X\times_ZY\to X\times_ZY'$ is a pullback of $Y\to Y'$).

For any stack $X$, open (resp.\ closed) embeddings $Y\hookrightarrow X$ are in natural bijection with open (resp.\ closed) subsets of $\left|X\right|$.

To check that a given map of spaces satisfies one of the properties $\cP$ above, it is often helpful to make use of the fact that these $\cP$ are all \emph{local on the target}, meaning that for every open cover $\{U_i\subseteq Y\}_i$, if $X\times_YU_i\to U_i$ has $\cP$ for every $i$, then so does $X\to Y$.
This leads to the following generalization for maps of stacks: if $F\to G$ is a representable morphism of stacks and $G'\to G$ is a representable morphism of stacks admitting local sections, then $F\to G$ has $\cP$ iff $F\times_GG'\to G'$ has $\cP$.
In fact, in this statement we need not assume that $G'\to G$ is representable, just that it admit local sections in the generalized sense that for every map $X\to G$ from a topological space $X$, there exists an open cover $\{U_i\subseteq X\}_i$ such that each $G'\times_GU_i\to U_i$ admits a section.
We thus say that ``$\cP$ descends along maps admitting local sections''.
The same descent property holds for representability itself:

\begin{lemma}[Representability descends under maps admitting local sections]\label{repdescend}
Let $F\to G$ be a map of stacks, and let $G'\to G$ be a map of stacks admitting local sections.
If $F\times_GG'\to G'$ is representable, then so is $F\to G$.
\end{lemma}

\begin{proof}
By replacing $F\to G$ and $G'\to G$ with their pullbacks under $U\to G$ for a topological space $U$, we may assume without loss of generality that $G$ is representable.
Since $G'\to G$ admits local sections, we may replace $G'\to G$ with the composition $\bigsqcup_iU_i\to G'\to G$ where $\{U_i\subseteq G\}_i$ is an open cover.
Now each $F\times_GU_i$ is representable by assumption, and gluing these spaces together on their common overlaps $F\times_G(U_i\cap U_j)$ gives a topological space representing $F$.
\end{proof}

A complex vector bundle over a stack $X$ is a representable map $V\to X$ together with maps $V\times_XV\to V$ and $\CC\times V\to V$ (both over $X$) such that for every map $U\to X$ from a topological space $U$, there exists an open cover $\{U_i\subseteq U\}_i$ and integers $n_i\geq 0$ such that $V\times_XU_i\to U_i$ is isomorphic to $\CC^{n_i}\times U_i\to U_i$ equipped with its fiberwise addition and scaling maps.
A pullback of a complex vector bundle is naturally a complex vector bundle.

The class of so called \emph{topological stacks} (those which admit a presentation via a topological groupoid) are somewhat better behaved than general stacks.
A \emph{topological groupoid} $M\righttwoarrows O$ consists of a pair of topological spaces $O$ (`objects') and $M$ (`morphisms'), two maps $M\righttwoarrows O$ (`source' and `target'), a map $O\to M$ (`identity'), an involution $M\to M$ (`inverse'), and a map $M\times_OM\to M$ (`composition') satisfying the axioms of a groupoid.
A topological groupoid $M\righttwoarrows O$ presents a stack $[M\righttwoarrows O]$ defined as follows.
An object of $[M\righttwoarrows O](X)$ is an open cover $\{U_i\subseteq X\}_i$ together with maps $U_i\to O$ and $U_i\cap U_j\to M$ satisfying a compatibility condition, and an isomorphism in $[M\righttwoarrows O](X)$ consists of maps $U_i\cap U_{i'}'\to M$ satisfying a compatibility condition.
There is a natural map $O\to[M\righttwoarrows O]$ (take the trivial open cover $\{X\subseteq X\}$) and the fiber product $O\times_{[M\righttwoarrows O]}O$ is naturally identified with $M$.
The morphism $O\to[M\righttwoarrows O]$ admits local sections (by definition), so since $O\times_{[M\righttwoarrows O]}O\to O$ is representable, it follows by descent that $O\to[M\righttwoarrows O]$ is representable.
Conversely, a representable map $U\to X$ admitting local sections from a topological space $U$ to a stack $X$ determines a topological groupoid $U\times_XU\righttwoarrows U$ presenting $X$.
Indeed, the fiber product $U\times_XU$ is representable (since $U$ and $U\to X$ are), it admits two maps to $U$ (the two projections), an involution (exchanging the two factors), and a composition map $(U\times_XU)\times_U(U\times_XU)=U\times_XU\times_XU\to U\times_XU$ (forgetting the middle factor), and one can check using the stack property that the natural map $[U\times_XU\righttwoarrows U]\to X$ is an equivalence.
A stack $X$ for which there exists a representable map $U\to X$ admitting local sections from a space $U$ is called \emph{topological}, and such $U\to X$ is called an \emph{atlas} for $X$.

For a topological stack $X$ with atlas $U\to X$, for any property $\cP$ which descends along maps admitting local sections, the map $U\to X$ has $\cP$ iff both maps $U\times_XU\to U$ have $\cP$, and the diagonal $X\to X\times X$ has $\cP$ iff the map $U\times_XU\to U\times U$ has $\cP$.
In particular, for any topological stack $X$, the diagonal $X\to X\times X$ is representable, and thus every map $Z\to X$ from a topological space $Z$ is representable.
More generally, for any map of topological stacks $X\to Y$, the relative diagonal $X\to X\times_YX$ is representable (by descent from its pullback $V\times_XV\to V\times_YV$ for an atlas $V\to X$).
If $X\to Y$ is representable and $Y$ is topological, then so is $X$ (if $U\to Y$ is an atlas for $Y$, then its pullback $V=U\times_YX\to X$ is an atlas for $X$), and the relative diagonal $X\to X\times_YX$ has $\cP$ iff the relative diagonal $V\to V\times_UV$ of atlases has $\cP$ (the latter is the pullback of the former under the map $V\times_UV\to X\times_YX$, which is representable and admits local sections since it is a pullback of $U\to Y$).

For a topological space $V$, a topological group $G$, and a continuous group action $G\acts V$, we may consider the action groupoid $G\times V\righttwoarrows V$ with source and target maps $(g,x)\mapsto x$ and $(g,x)\mapsto gx$.
The stack associated to this groupoid is denoted $V/G$ and is called the stack quotient of the action $G\acts V$.
If $G$ is discrete, the two maps $G\times V\to V$ are \'etale, so $V\to V/G$ is \'etale.
If $G$ is compact and $V$ is Hausdorff, the map $G\times V\to V\times V$ is universally closed (factor as $G\times V\to G\times V\times V$ which is a closed embedding since $V$ is Hausdorff and $G\times V\times V\to V\times V$ which is universally closed since $G$ is compact), so $V/G$ has universally closed diagonal.
If $G$ is Hausdorff, then the map $G\times V\to V\times V$ is separated, so $V/G$ has separated diagonal.
Thus if $G$ is compact Hausdorff and $V$ is Hausdorff then $V/G$ has proper diagonal.

A groupoid presentation of a topological stack $X$ also gives a description of its coarse space $\left|X\right|$ as follows.
For an atlas $U\to X$, consider the equivalence relation $\sim_X$ on $U$ given by the image of $U\times_XU\to U\times U$.
There is a map $X\to U/{\sim_X}$ (which is tautological once we regard $X$ as $[U\times_XU\righttwoarrows U]$), inducing a map $\left|X\right|\to U/{\sim_X}$, which is a bijection, essentially by definition.
Since an open substack of $X$ pulls back to an open subset of $U$ invariant under $\sim_X$, it follows that $\left|X\right|\xrightarrow\sim U/{\sim_X}$ is open and is thus a homeomorphism.
In particular, it follows that the coarse space $\left|V/G\right|$ of the stack quotient $V/G$ is the usual topological quotient of $V$ by $G$.

An action $G\acts V$ is called \emph{locally trivial} iff $V$ admits a cover by $G$-invariant open sets $\{G\times Z_i\subseteq V\}_i$ where $G\acts G\times Z_i$ acts by left multiplication on $G$ and trivially on $Z_i$.
If $G\acts V$ is locally trivial, then the natural map from the stack quotient to the topological quotient $V/G\to\left|V/G\right|$ is an equivalence.
Indeed, this assertion is local on $\left|V/G\right|$, so it suffices to consider the case of $G\acts G\times Z$, where it holds by inspection.

\begin{definition}\label{separatedorbsipacedefn}
A \emph{(separated) orbispace} is a stack $X$ for which there exists a representable \'etale surjection $U\to X$ from a topological space $U$ (an `\'etale atlas'), and the diagonal $X\to X\times X$ is proper.
\end{definition}

\begin{proposition}\label{orbispacelocal}
A stack $X$ is an orbispace iff $\left|X\right|$ is Hausdorff and there exists an open cover $\{V_i/G_i\subseteq X\}_i$ where $G_i$ are finite discrete groups acting on Hausdorff spaces $V_i$.
\end{proposition}

(Similar results include \cite[Theorem 1.108]{behrendintroduction} and \cite[Proposition 14.10]{noohifoundations}.)

\begin{proof}
Let $X$ be an orbispace, and let us show that there is an open cover $\{V_i/G_i\subseteq X\}_i$.
Fix an \'etale atlas $U\to X$, and let $u\in U$.
The automorphism group $G_u:=\{u\}\times_X\{u\}\subseteq U\times_XU$ is finite and discrete since $X\to X\times X$ is proper.
Since the two projections $U\times_XU\righttwoarrows U$ are \'etale, for every $g\in G$ there exists an open neighborhood $g\in U_g\subseteq U\times_XU$ such that each projection restricted to $U_g$ is an open embedding (this gives another proof that $G$ has the discrete topology).
Since $G$ is finite and $U\times_XU\to U\times U$ is separated, we may take these $U_g$ to be disjoint.
Now the complement of $\bigsqcup_gU_g$ is closed and disjoint from $G$, so it projects to a closed (by properness) subset of $U\times U$ disjoint from $(u,u)$.
Hence there is an open neighborhood $u\in V\subseteq U$ such that $V\times_XV\subseteq\bigsqcup_gU_g$.
Thus $V\times_XV$ is a disjoint union of pieces indexed by $G$, and each piece maps homeomorphically to $V$ under each projection.
By further shrinking $V$, we may assume that the map $V\times_XV\to G$ respects composition (this is possible since composition is continuous).
It follows that $V\times_XV\righttwoarrows V$ is an action groupoid $G\times V\righttwoarrows V$ for an action $G\acts V$.
Since $V=\{1\}\times V\hookrightarrow G\times V=V\times_XV\to V\times V$ expresses the diagonal of $V$ as a composition of proper maps, we conclude that $V$ is Hausdorff.
Now the map of groupoids $(G\times V\righttwoarrows V)=(V\times_XV\righttwoarrows V)\to(U\times_XU\righttwoarrows U)$ induces a map of stacks $V/G\to X$.
To see that this is an open embedding, let $V^+\subseteq U$ denote the orbit of $V$ under the morphisms $U\times_XU\righttwoarrows U$.
Since the projections $U\times_XU\righttwoarrows U$ are \'etale, it follows that $V^+\subseteq U$ is open, and hence $[V^+\times_XV^+\righttwoarrows V^+]\to[U\times_XU\righttwoarrows U]$ is an open embedding of stacks; denote by $Z\subseteq X$ this open substack, so $V\to Z$ is an \'etale atlas.
Thus the topological groupoid $V\times_XV=V\times_ZV\righttwoarrows V$ presents $Z$, so $V/G\to X$ is an open embedding as desired.

Let us now show that for any orbispace $X$, its coarse space $\left|X\right|$ is Hausdorff.
We saw earlier that $\left|X\right|$ is the quotient of $U$ by the image of $U\times_XU\to U\times U$ (which is an equivalence relation).
This equivalence relation is closed since $U\times_XU\to U\times U$ is proper, so $\left|X\right|$ is Hausdorff provided the quotient map $U\to\left|X\right|$ is open.
Now openness of $U\to\left|X\right|$ does not depend on which atlas $U\to X$ we take: there are \'etale maps $U\leftarrow U\times_XU'\to U'$ over $\left|X\right|$, which means that $U\to\left|X\right|$ is open iff $U'\to\left|X\right|$ is open.
Moreover, openness of $U\to\left|X\right|$ can be checked locally on $\left|X\right|$, so we may assume without loss of generality that $X=V/G$.
Hence it is enough to note that the quotient map $V\to\left|V/G\right|$ (induced by the canonical \'etale atlas $V\to V/G$) is open.
Thus $\left|X\right|$ is Hausdorff.

Finally, let us show that if $\left|X\right|$ is Hausdorff and there is an open cover $\{V_i/G_i\subseteq X\}_i$, then $X$ is an orbispace.
The maps $V_i\to V_i/G_i$ are representable \'etale, so $U:=\bigsqcup_iV_i\to X$ is an \'etale atlas.
To show that the diagonal $X\to X\times X$ is proper, it is equivalent to show that $U\times_XU\to U\times U$ is proper.
This reduces to showing that $V\times_XV'\to V\times V'$ is proper for any pair $V/G\hookrightarrow X\hookleftarrow V'/G'$.
Since $\left|X\right|$ is Hausdorff, the map $V\times_{\left|X\right|}V'\to V\times V'$ a closed embedding (as it is a pullback of the diagonal of $\left|X\right|$), so it follows that $V\times_XV'\to V\times V'$ is proper iff $V\times_XV'\to V\times_{\left|X\right|}V'$ is proper.
Now for the purposes of studying the latter map $V\times_XV'\to V\times_{\left|X\right|}V'$, we may as well shrink $V$, $V'$, and $X$ so that $V/G=X=V'/G'$.
Now the diagonal of $V/G=V'/G'=X$ is proper, hence so is $V\times_XV'\to V\times V'$.
\end{proof}

\begin{corollary}\label{hausdorffatlas}
Every orbispace $X$ has an \'etale atlas $U\to X$ with $U$ Hausdorff; equivalently, every \'etale atlas $U\to X$ has $U$ locally Hausdorff.
\end{corollary}

\begin{proof}
By Proposition \ref{orbispacelocal} there is an open cover $\{V_i/G_i\subseteq X\}_i$ with $V_i$ Hausdorff, so $U:=\bigsqcup_iV_i\to X$ is an \'etale atlas with $U$ Hausdorff.
Now for any two \'etale atlases $U,U'\to X$, consideration of the surjective \'etale maps $U\leftarrow U\times_XU'\to U'$ shows that $U$ is locally Hausdorff iff $U'$ is locally Hausdorff.
Given any \'etale atlas $U\to X$ with $U$ locally Hausdorff and any open cover $\{U_i\subseteq U\}_i$ with $U_i$ Hausdorff, the disjoint union $U':=\bigsqcup_iU_i\to X$ is an \'etale atlas with $U'$ Hausdorff.
\end{proof}

\begin{corollary}\label{atlasrefine}
For any \'etale atlas $U\to X$ on an orbispace, there exists an open cover $\{V_i/G_i\subseteq X\}_i$ as in Proposition \ref{orbispacelocal} such that each map $V_i\to X$ factors through an open embedding $V_i\to U$.
\end{corollary}

\begin{proof}
Let $U\to X$ be given and fix any open cover $\{V_i/G_i\subseteq X\}_i$ as in Proposition \ref{orbispacelocal}.
Since $U\times_XV_i\to V_i$ is \'etale, it admits local sections, and hence by replacing each $V_i$ with an open cover of itself, we may assume that each projection $U\times_XV_i\to V_i$ admits a section (which is thus an open embedding).
Now the resulting maps $V_i\to U$ are \'etale by the factorization $V_i\to U\times_XV_i\to U$, so by again replacing each $V_i$ with an open cover, we may assume they are open embeddings.

Alternatively, we could note that the open cover $\{V_i/G_i\subseteq X\}_i$ produced by the proof of Proposition \ref{orbispacelocal} is in fact of the desired form.
\end{proof}

\begin{corollary}\label{criterionrepresentability}
A map of orbispaces $X\to Y$ is representable iff it is injective on isotropy groups.
In particular, an orbispace $X$ is a space iff it has trivial isotropy.
\end{corollary}

\begin{proof}
Any representable morphism of stacks is injective on isotropy groups (just test against points).
Thus we are left with showing that a map of orbispaces $X\to Y$ which is injective on isotropy groups is representable.

Since representability descends under maps admitting local sections, it suffices to show that the fiber product $X'=X\times_YY'$ is representable for some \'etale atlas $Y'\to Y$.
Note that $X'$ has trivial isotropy since $Y'$ has trivial isotropy and $X'\to Y'$ is injective on isotropy groups (being a pullback of $X\to Y$).

We claim that $X'$ is an orbispace, provided we take $Y'$ to be Hausdorff (which we can by Corollary \ref{hausdorffatlas}).
The pullback of an \'etale atlas $U\to X$ is an \'etale atlas $U'\to X'$ (since $X'\to X$ is representable, being a pullback of $Y'\to Y$).
The diagonal of $X'$ is the composition $X'\to X'\times_XX'\to X'\times X'$.
The second map $X'\times_XX'\to X'\times X'$ is proper as it is a pullback of $X\to X\times X$.
To analyze the first map $X'\to X'\times_XX'$, note that it pulls back to $U'\to U'\times_UU'$ under the map $U'\times_UU'\to X'\times_XX'$ which admits local sections (being a pullback of $U\to X$).
Since $Y\to Y\times Y$ is separated, its pullback $Y'\times_YY'\to Y'\times Y'$ is also separated, which implies each projection $Y'\times_YY'\righttwoarrows Y'$ is separated since $Y'$ is Hausdorff, which implies $Y'\to Y$ is separated.
Hence its pullback $U'\to U$ is separated, so $U'\to U'\times_UU'$ is a closed embedding, hence proper.
Thus $X'\to X'\times_XX'$ is proper, and we conclude that $X'$ is an orbispace.

We are thus reduced to showing that an orbispace $Z$ with trivial isotropy is a space.
By Proposition \ref{orbispacelocal}, we know that $Z$ is given locally by $V/G$ for $V$ Hausdorff and $G$ finite discrete acting freely (since $Z$ has trivial isotropy).
Free actions $G\acts V$ with $V$ Hausdorff and $G$ finite are locally trivial, so we conclude that the map $Z\to\left|Z\right|$ is an equivalence.
Alternatively, we could note that the chart $V/G$ near a given $x\in X$ constructed in the proof of Proposition \ref{orbispacelocal} in fact satisfies $G=G_x$ by definition.
\end{proof}

\section{Coverings and nerves}\label{covernerve}

We show how Theorem \ref{mainsimplicial} implies Theorem \ref{maingeneral}.
It is enough to show that a given orbispace admits a representable map to a simplicial complex of groups; this is Proposition \ref{maptosimplicial}.
This simplicial complex of groups is basically just the nerve of a suitable open cover, however its construction is somewhat more delicate than one might initially expect.

A \emph{sieve} on a topological space $X$ is a subset $S\subseteq 2^X$ consisting of open sets such that $U'\subseteq U\in S$ implies $U'\in S$.
A \emph{covering sieve} on $X$ is a sieve $S$ such that $\bigcup_{U\in S}U=X$.
An open cover $\{U_i\subseteq X\}_i$ is said to \emph{generate} the covering sieve on $X$ consisting of those open sets which are contained in some $U_i$.

\begin{definition}
A \emph{connection sieve} on a map of spaces $f:X\to Y$ is a covering sieve $S$ on $X$ such that (1) for $U\in S$, the composition $U\to X\to Y$ is an open embedding, and (2) for $U,V\in S$ with $f(U)=f(V)$, either $U=V$ or $U\cap V=\varnothing$.
\end{definition}

Note that for sieves satisfying condition (1), condition (2) is equivalent to condition (2$'$) for $U,V\in S$ either $U\cap V=\varnothing$ or $f(U\cap V)=f(U)\cap f(V)$.
If $S'\subseteq S$ is an inclusion of covering sieves and $S$ is a connection sieve on $X\to Y$, then so is $S'$.
In particular, if $S$ and $S'$ are connection sieves on $X\to Y$ then so is $S\cap S'$.
To check that an open cover $\{U_i\subseteq X\}_i$ generates a connection sieve, it is enough to check axioms (1) and (2$'$) for the open sets $U_i$.

\begin{definition}
A map $X\to Y$ is called \emph{strongly \'etale} iff it admits a connection sieve.
\end{definition}

Open embeddings are strongly \'etale, and strongly \'etale maps are separated and \'etale (however the converse is false by Example \ref{nonstronglyetale} below).
Being strongly \'etale is preserved under pullback (take the connection sieve generated by the pullback of the original connection sieve), and the class of strongly \'etale maps is closed under composition (a connection sieve $S_{X/Z}$ for a composition $X\to Y\to Z$ is given by those elements of a fixed connection sieve $S_{X/Y}$ for $X\to Y$ whose image lies in a fixed connection sieve $S_{Y/Z}$ for $Y\to Z$).
A disjoint union $\bigsqcup_iX_i\to\bigsqcup_iY_i$ of strongly \'etale maps $X_i\to Y_i$ is strongly \'etale (take disjoint union of connection sieves), and the projection $A\times Y\to Y$ is strongly \'etale for any discrete space $A$.
Being strongly \'etale is \emph{not} local on the target, as shown by the following example.

\begin{example}\label{nonstronglyetale}
Let us construct a finite covering space which is not strongly \'etale.
Let $Z=\{1,\frac 12,\frac 14,\frac 18,\cdots\}\cup\{0\}\subseteq\RR$ with the subspace topology.
Every double cover of $Z$ or $Z\setminus 0$ is trivial.
A given double cover of $Z\setminus 0$ has, however, many distinct extensions to $Z$, indexed by functions $Z\setminus 0\to\ZZ/2$ modulo those functions which extend continuously to zero.

Now suppose given a double cover of $Z$ together with a connection sieve on it.
The restriction of this data to $Z\setminus 0$ remembers the extension of the double cover to $Z$ (use the elements of the connection sieve projecting to open sets of the form $Z\cap(0,\varepsilon)$).

Let $\tilde Z^+_\alpha\to Z^+$ denote the double cover obtained by gluing two copies of $Z\sqcup Z\to Z$ over $Z\setminus 0$ via some map $\alpha:Z\setminus 0\to\ZZ/2$.
If $\tilde Z^+_\alpha\to Z^+$ is strongly \'etale, then we can take any connection sieve on it, restrict to the common $Z\setminus 0\subseteq Z^+$, and deduce that the two extensions of the double cover of $Z\setminus 0$ to $Z$ coincide, in other words that $\alpha$ extends continuously to zero.
Hence the double cover $\tilde Z^+_\alpha\to Z^+$ is strongly \'etale iff $\alpha$ extends continuously to zero.
\end{example}

Recall that a topological space is called \emph{paracompact} iff every open cover admits a locally finite refinement \cite{dieudonne}.
If $X$ is paracompact and Hausdorff, then there exists a \emph{partition of unity} subordinate to any given locally finite open cover $\{U_i\subseteq X\}_i$, namely functions $f_i:X\to\RR_{\geq 0}$ with $\supp f_i\subseteq U_i$ such that $\sum_if_i\equiv 1$ (recall that the support $\supp f$ of a function $f:X\to\RR_{\geq 0}$ is by definition the complement of the largest open set over which $f$ vanishes identically).

\begin{lemma}\label{connectionparacompact}
If $Y$ is paracompact Hausdorff, then a map $X\to Y$ is strongly \'etale iff there exists an open cover $\{U_i\subseteq Y\}_i$ such that each map $X\times_YU_i\to U_i$ is strongly \'etale.
\end{lemma}

\begin{proof}
Fix an open cover $\{U_i\subseteq Y\}_i$ and connection sieves $S_i$ on $X\times_YU_i\to U_i$.
Since $Y$ is paracompact, we may assume that our open cover $\{U_i\subseteq Y\}_i$ is locally finite.
Using a partition of unity subordinate to this open cover, we may find another open cover $\{V_I\subseteq Y\}_I$ indexed by non-empty finite subsets $I$ of the original index set, such that $V_I\subseteq\bigcap_{i\in I}U_i$ and $V_I\cap V_J=\varnothing$ unless $I\subseteq J$ or $J\subseteq I$ (explicitly, we may take $V_I$ to be the locus where $\min_{i\in I}f_i>\max_{i\notin I}f_i$).
We may now define a connection sieve on $X\to Y$ as the union over $I$ of $2^{X\times_YV_I}\cap\bigcap_{i\in I}S_i$.
\end{proof}

An orbispace will be called paracompact (resp.\ coarsely finite-dimensional, $d$-dimensional) iff its coarse space is.

\begin{proposition}\label{atlasconnection}
Every paracompact orbispace $X$ has an \'etale atlas $U\to X$ for which the projections $U\times_XU\righttwoarrows U$ are strongly \'etale.
In fact, there exists such $U$ of the form $U=\bigsqcup_iV_i$ for an open cover $\{V_i/G_i\subseteq X\}_i$ as in Proposition \ref{orbispacelocal}.
\end{proposition}

\begin{proof}
Fix an open cover $\{V_i/G_i\hookrightarrow X\}_i$ as in Proposition \ref{orbispacelocal}.
Since $\left|X\right|$ is paracompact, we may shrink the spaces $V_i$ ($G_i$-equivariantly) so as to ensure that the associated open cover $\{\left|V_i/G_i\right|\subseteq\left|X\right|\}_i$ of coarse spaces is locally finite.
Choose a partition of unity $\{f_i:\left|X\right|\to\RR_{\geq 0}\}_i$ subordinate to the open cover $\{\left|V_i/G_i\right|\subseteq\left|X\right|\}_i$.
Let $V_i^0\subseteq V_i$ denote the open subset where $f_i>0$.
We will show that the \'etale atlas $U:=\bigsqcup_iV_i^0$ has the desired property.

It suffices to show that for any pair of open embeddings $V/G\hookrightarrow X\hookleftarrow W/H$ and pair of functions $f_V,f_W:\left|X\right|\to\RR_{\geq 0}$ supported inside $\left|V/G\right|$ and $\left|W/H\right|$, respectively, the projection $W^0\times_XV^0\to V^0$ is strongly \'etale.
We begin by considering the map $W\times_XV\to V$, which is a finite covering space over its open image $W/H\times_XV\subseteq V$ (by pullback from $W\to W/H$, which is a finite covering space by descent from its pullback $H\times W\to W$).
Thus every point of $W/H\times_XV$ has a neighborhood over which $W\times_XV\to V$ is strongly \'etale.
Even better, since $V$ is Hausdorff and $G$ is finite, each $G$-orbit inside $W/H\times_XV$ has a neighborhood over which $W\times_XV\to V$ is strongly \'etale.
In other words, each point of $\left|W/H\right|\cap\left|V/G\right|\subseteq\left|X\right|$ has a neighborhood over which $W\times_XV\to V$ is strongly \'etale.

Now since $\left|X\right|$ is paracompact, there exists a locally finite open cover of $\left|X\right|$ by $\left|X\right|\setminus(\supp f_V\cup\supp f_W)$ together with open subsets $\{A_i\subseteq\left|W/H\right|\cap\left|V/G\right|\}_i$ over which $W\times_XV\to V$ is strongly \'etale.
Fix a partition of unity $g:\left|X\right|\to\RR_{\geq 0}$ supported inside $\left|X\right|\setminus(\supp f_V\cup\supp f_W)$ and $\{g_i:\left|X\right|\to\RR_{\geq 0}\}$ supported inside $A_i$, that is $g+\sum_ig_i\equiv 1$.
Now the patching procedure for connection sieves from the proof of Lemma \ref{connectionparacompact} shows that $W\times_XV\to V$ is strongly \'etale over the complement of $\supp g$.
In particular, it follows by restriction that $W^0\times_XV^0\to V^0$ is strongly \'etale.
\end{proof}

A \emph{simplicial complex} is a pair $X=(V,S)$ consisting of a set $V$ (``vertices'') and a set $S\subseteq 2^V\setminus\{\varnothing\}$ (``simplices'') of finite subsets of $V$ such that $S$ contains all singletons and $\varnothing\ne A\subseteq B\in S$ implies $A\in S$.
The \emph{star} $\st(X,\sigma)\subseteq X$ of a simplex $\sigma$ in a simplicial complex $X$ is the subcomplex consisting of all simplices $\tau\subseteq X$ with $\sigma\cup\tau\in S(X)$.

A map of simplicial complexes $X\to Y$ is a map of vertex sets $V(X)\to V(Y)$ which maps simplices to simplices (the image of an element of $S(X)$ is an element of $S(Y)$).
A map of simplicial complexes is called injective iff the map on vertex sets (hence also the map on simplices) is injective.
A map of simplicial complexes $f:X\to Y$ is called \emph{\'etale} (resp.\ \emph{locally injective}) iff the induced maps on stars $\st(X,\sigma)\to\st(Y,f(\sigma))$ are isomorphisms (resp.\ injective).
We will call a map of simplicial complexes $X\to Y$ \emph{sufficiently \'etale} iff every simplex $\sigma\subseteq Y$ (equivalently, every vertex) is the image of a simplex $\tau\subseteq X$ at which $X\to Y$ is \'etale (this is a useful weakening of the condition of being surjective and \'etale, which in the context of simplicial complexes is too strong).

The \emph{geometric realization} $\left\|X\right\|$ of a simplicial complex $X$ is the set of tuples $t\in\RR_{\geq 0}^{V(X)}$ with $\sum_vt_v=1$ such that $\{v:t_v>0\}\in S(X)$, topologized by declaring that the realization of the complete simplex on $k+1$ vertices has the usual topology and that a realization $\left\|X\right\|$ is given the strongest topology for which (the realization of) every map from a complete simplex to $X$ is continuous.
The geometric realization of an \'etale map of simplicial complexes is an \'etale map of spaces.

A \emph{locally injective simplicial complex groupoid} $M\righttwoarrows O$ consists of simplicial complexes $O$ and $M$ together with structure maps satisfying the axioms of a groupoid, where both maps $M\righttwoarrows O$ are locally injective.
Local injectivity of the two maps $M\righttwoarrows O$ implies that the natural map $\left\|M\times_OM\right\|\xrightarrow\sim\left\|M\right\|\times_{\left\|O\right\|}\left\|M\right\|$ is a homeomorphism, and thus the geometric realization $\left\|M\right\|\righttwoarrows\left\|O\right\|$ is a topological groupoid.
If $\partial O\subseteq O$ denotes the subcomplex consisting of those simplices $\sigma\subseteq O$ for which it is not the case that the first projection $M\to O$ is \'etale at every simplex $\tau\subseteq M$ mapped to $\sigma$ under the second projection, then the natural map $\left\|O\right\|\setminus\left\|\partial O\right\|\to[\left\|M\right\|\righttwoarrows\left\|O\right\|]$ is \'etale.
A locally injective simplicial complex groupoid $M\righttwoarrows O$ is called \emph{sufficiently \'etale} iff this map is surjective (equivalently, every vertex of $O$ is $M$-isomorphic to one not in $\partial O$).

The \emph{abstract simplex category} $\Simp$ has objects finite totally ordered sets and has morphisms weakly order preserving maps; every object of $\Simp$ is isomorphic to $[n]:=\{0<\cdots<n\}$ for a unique integer $n\geq 0$.
A simplicial object in a category $\C$ is a functor $\Simp^\op\to\C$, and the category of simplicial objects in $\C$ is denoted $\s\C$.
If $\C$ is complete (resp.\ cocomplete) then so is $\s\C$, and limits (resp.\ colimits) are calculated pointwise.

We will consider only simplicial sets (objects of the category $\s\Set$) and simplicial groupoids (objects of the category $\s\Grpd$).
A simplicial set (or groupoid) will be denoted $X_\bullet$, where $X_n$ is its set (or groupoid) of $n$-simplices.
We denote by $\Delta^n_\bullet\in\s\Set$ the standard $n$-simplex, given by $[m]\mapsto\Hom([m],[n])$.
The Yoneda lemma implies that $X_n=\Hom(\Delta^n_\bullet,X_\bullet)$.

A map of simplicial sets $X_\bullet\to Y_\bullet$ is called injective iff it is so levelwise (i.e.\ every $X_n\to Y_n$ is injective).
A map $X_\bullet\to Y_\bullet$ is called \'etale (resp.\ locally injective) iff for every map $[n]\to[m]$ in $\Simp$, the induced map $X_m\xrightarrow\sim Y_m\times_{Y_n}X_n$ is a bijection (resp.\ injective) (it is equivalent to impose this condition only for $n=0$).
A map $X_\bullet\to Y_\bullet$ is called \'etale (resp.\ locally injective) at a given $n$-simplex $\sigma$ of $X_\bullet$ iff $X_m\xrightarrow\sim Y_m\times_{Y_n}\{\sigma\}$ is a bijection (resp.\ injective) (this condition at a given $\sigma$ implies the same at any preimage of $\sigma$ under any structure map of $X_\bullet$).
A map $X_\bullet\to Y_\bullet$ is called sufficiently \'etale iff the $n$-simplices of $X_\bullet$ at which the map is \'etale surject onto the $n$-simplices of $Y_\bullet$ (it is equivalent to impose this condition only for $n=0$).
These notions generalize to maps of simplicial groupoids by replacing `injectivity' and `surjectivity' for maps of sets with `full faithfulness' and `essential surjectivity' for functors of groupoids.
These properties are all preserved under pullback and closed under composition.

The \emph{geometric realization} $\left\|X_\bullet\right\|$ of a simplicial set $X_\bullet$ is the colimit of $\Delta^n$ over all maps $\Delta^n_\bullet\to X_\bullet$.
Geometric realization is cocontinuous, and the natural map $\left\|\lim_\alpha(X_\bullet)_\alpha\right\|\to\lim_\alpha\left\|(X_\bullet)_\alpha\right\|$ is bijective, however it need not be a homeomorphism even for finite limits (for example, the case of binary products is discussed in \cite[Theorem 14.3, Remark 14.4]{maysimplicial} and \cite{whiteheadcwproduct,milnorcwproduct,brooketaylorcwproduct}).
The map $\left\|X_\bullet\times_{Y_\bullet}Z_\bullet\right\|\to\left\|X_\bullet\right\|\times_{\left\|Y_\bullet\right\|}\left\|Z_\bullet\right\|$ is a homeomorphism if at least one of the maps $X_\bullet\to Y_\bullet$ and $Z_\bullet\to Y_\bullet$ is locally injective.
Thus a locally injective simplicial set groupoid $M_\bullet\righttwoarrows O_\bullet$ determines a topological groupoid $\left\|M_\bullet\right\|\righttwoarrows\left\|O_\bullet\right\|$.

Let us now introduce the geometric realization $\left\|X_\bullet\right\|$ of any simplicial groupoid $X_\bullet$ which is \emph{\'etale}, meaning that it admits a locally injective sufficiently \'etale map $U_\bullet\to X_\bullet$ from a simplicial set $U_\bullet$.
For any simplicial groupoid $X_\bullet$ and any locally injective map $U_\bullet\to X_\bullet$, the pair of simplicial sets $U_\bullet\times_{X_\bullet}U_\bullet\righttwoarrows U_\bullet$ forms a locally injective simplicial set groupoid, whose geometric realization $\left\|U_\bullet\times_{X_\bullet}U_\bullet\right\|\righttwoarrows\left\|U_\bullet\right\|$ thus defines a topological stack.
The geometric realization of $X_\bullet$ is defined as this topological stack $[\left\|U_\bullet\times_{X_\bullet}U_\bullet\right\|\righttwoarrows\left\|U_\bullet\right\|]$ associated to any locally injective sufficiently \'etale map $U_\bullet\to X_\bullet$.

\begin{lemma}\label{etalesimplicialgroupoidrealize}
The geometric realization of an \'etale simplicial groupoid $X_\bullet$ is well defined.
\end{lemma}

\begin{proof}
Let $U_\bullet,U'_\bullet\to X_\bullet$ be locally injective and sufficiently \'etale.
Let $U''_\bullet:=U_\bullet\times_{X_\bullet}U_\bullet'\to X_\bullet$, and consider the map of simplicial set groupoids $(U''_\bullet\times_{X_\bullet}U''_\bullet\righttwoarrows U''_\bullet)\to(U_\bullet\times_{X_\bullet}U_\bullet\righttwoarrows U_\bullet)$.
Since $U''_\bullet\to U_\bullet$ is locally injective and sufficiently \'etale, it follows that this map induces an isomorphism of topological stacks, and the same applies to $U'_\bullet$ in place of $U_\bullet$.
\end{proof}

\begin{lemma}\label{etalesimplicialgroupoidorbispace}
The geometric realization $\left\|X_\bullet\right\|$ of an \'etale simplicial groupoid $X_\bullet$ with finite isotropy is an orbispace.
\end{lemma}

\begin{proof}
All geometric realizations are Hausdorff, so $\left\|U_\bullet\times_{X_\bullet}U_\bullet\right\|\to\left\|U_\bullet\right\|\times\left\|U_\bullet\right\|$ is separated, hence $\left\|X_\bullet\right\|$ has separated diagonal.
Since $X_\bullet$ has finite isotropy, the map $U_\bullet\times_{X_\bullet}U_\bullet\to U_\bullet\times U_\bullet$ has finite fibers, which combined with local injectivity of $U_\bullet\times_{X_\bullet}U_\bullet\righttwoarrows U_\bullet$ implies that $\left\|U_\bullet\times_{X_\bullet}U_\bullet\right\|\to\left\|U_\bullet\right\|\times\left\|U_\bullet\right\|$ is universally closed, hence $\left\|X_\bullet\right\|$ has universally closed diagonal.
We have thus shown that $\left\|X_\bullet\right\|$ has proper diagonal.

To construct an \'etale atlas for $\left\|X_\bullet\right\|$, let $\partial U_\bullet\subseteq U_\bullet$ denote the simplicial subset consisting of those simplices of $U_\bullet$ at which $U_\bullet\to X_\bullet$ is not \'etale.
Then $\left\|U_\bullet\right\|\setminus\left\|\partial U_\bullet\right\|\to\left\|X_\bullet\right\|$ is \'etale.
To see that it is surjective, it suffices to show that $\left\|U_\bullet\times_{X_\bullet}U_\bullet\right\|\setminus\left\|U_\bullet\times_{X_\bullet}\partial U_\bullet\right\|\to\left\|U_\bullet\right\|$ is surjective (note that surjectivity descends under surjective maps such as $\left\|U_\bullet\right\|\to\left\|X_\bullet\right\|$), and this follows since $U_\bullet\to X_\bullet$ is sufficiently \'etale.
\end{proof}

A simplicial complex $X$ gives rise to a simplicial set $b_\bullet X$ (its \emph{barycentric subdivision}) whose $n$-simplices are chains of simplices $\sigma_0\subseteq\cdots\subseteq\sigma_n\subseteq X$ (in other words, $b_\bullet X$ is the nerve of $S(X)$).
Barycentric subdivision preserves injectivity, local injectivity, \'etale, and sufficiently \'etale.
There is a natural identification of geometric realizations $\left\|X\right\|=\left\|b_\bullet X\right\|$.
Moreover, for a locally injective sufficiently \'etale simplicial complex groupoid $M\righttwoarrows O$, there is a natural identification $[\left\|M\right\|\righttwoarrows\left\|O\right\|]=\left\|[b_\bullet M\righttwoarrows b_\bullet O]\right\|$ (the simplicial groupoid $[b_\bullet M\righttwoarrows b_\bullet O]$ is \'etale since $M\righttwoarrows O$ is sufficiently \'etale).

A simplicial complex of groups $(Z,G)$ also admits a barycentric subdivision $b_\bullet(Z,G)$ which is a simplicial groupoid.
In fact, we will define $b_\bullet(Z,G)$ for any simplicial complex $Z$ equipped with a functor $G:S(Z)^\op\to\Grpd$ from the face poset to groupoids (a simplicial complex of groups $(Z,G)$ determines such a functor which sends $\sigma$ to $\BB G_\sigma$).
The groupoid of $n$-simplices in the barycentric subdivision $b_\bullet(Z,G)$ is now defined as the groupoid of functors from the category $0\to\cdots\to n$ to the category whose objects are pairs $\sigma\in S(Z)$ and $o\in G_\sigma$ and whose morphisms are inclusions $\sigma_1\subseteq\sigma_2$ covered by maps $o_1\to o_2|_{\sigma_1}$.
The barycentric subdivision of a simplicial complex of groups is \'etale, as can be seen as follows.
For any $\sigma\subseteq Z$ and $o\in G_\sigma$, we consider the functor $S(\st(Z,\sigma))^\op\to\Grpd$ given by $\tau\mapsto G_{\sigma\cup\tau}\times_{G_\sigma}\{o\}$.
Applying the nerve construction from just above to this functor, we obtain a simplicial set mapping to $b_\bullet(Z,G)$, which is the required locally injective map which is \'etale over $(\sigma,o)\in b_0(Z,G)$.
An essentially equivalent discussion (albeit without barycentrically subdividing) appears in \cite[12.24--12.25]{bridsonhaefliger}.
Since $b_\bullet(Z,G)$ is \'etale, it has a geometric realization $\left\|b_\bullet(Z,G)\right\|$ which we also write as $\left\|(Z,G)\right\|$.
When $G$ are finite groups (or, more generally, groupoids with finite isotropy), then the geometric realization $\left\|(Z,G)\right\|$ is an orbispace by Lemma \ref{etalesimplicialgroupoidorbispace}, and this is what we have been calling the orbispace presented by the simplicial complex of finite groups $(Z,G)$.

\begin{lemma}\label{MOgroup}
Let $M\righttwoarrows O$ be a locally injective simplicial complex groupoid with the following properties:
\begin{itemize}
\item The vertices of every simplex of $O$ are pairwise non-isomorphic via $M$.
\item If simplices $\sigma$ and $\sigma'$ of $O$ have vertex sets which are isomorphic via $M$, then $\sigma$ and $\sigma'$ are themselves isomorphic via $M$.
\end{itemize}
Then there is a simplicial complex of groups giving rise to the same simplicial groupoid as $M\righttwoarrows O$.
\end{lemma}

\begin{proof}
The hypotheses imply that there is a simplicial complex $Z$ whose vertices are the isomorphism classes in the vertex groupoid $V(M)\righttwoarrows O(M)$, and whose simplices are the $M$-isomorphism classes of simplices of $O$.
Now $M\righttwoarrows O$ defines a functor $G:S(Z)^\op\to\Grpd$, and there is a natural isomorphism between $b_\bullet(Z,G)$ and the simplicial groupoid $[b_\bullet M\righttwoarrows b_\bullet O]$.
By definition, all the groupoids $G_\sigma$ have a single isomorphism class, and all the functors $G_\tau\to G_\sigma$ are faithful (this follows from local injectivity of $M\righttwoarrows O$).
Choosing (independently) a base object of each $G_\sigma$ shows $(Z,G)$ comes from a simplicial complex of groups.
\end{proof}

The \emph{nerve} $N(X,\{U_i\}_i)$ of a collection of open sets $\{U_i\subseteq X\}_i$ is the simplicial complex whose vertices $V$ are the indices $i$ with $U_i\ne\varnothing$, and in which a collection $I$ of indices spans a simplex (i.e.\ $I\in S$) iff $\bigcap_{i\in I}U_i\ne\varnothing$.
A partition of unity $\{f_i:X\to\RR_{\geq 0}\}_i$ subordinate to a locally finite open cover $\{U_i\subseteq X\}_i$ defines a map from $X$ to the geometric realization $\left\|N(X,\{U_i\}_i)\right\|$ of the nerve of the open cover.

\begin{proposition}\label{maptosimplicial}
For any paracompact orbispace $X$, there exists a simplicial complex of groups $(Z,G)$ and a map $X\to\left\|(Z,G)\right\|$ which is injective on isotropy groups.
Moreover, we may take $Z$ to be locally finite, and we may take the groups $G_z$ for vertices $z\in Z$ to be isotropy groups of points of $X$.
If $X$ is coarsely finite-dimensional (resp.\ $d$-dimensional), then $Z$ may be taken to be finite-dimensional (resp.\ $d$-dimensional).
\end{proposition}

\begin{proof}
We begin with an open cover $\{V_i/G_i\subseteq X\}_i$ with the properties guaranteed by Proposition \ref{atlasconnection}, and we set $U:=\bigsqcup_iV_i$.
Since $\left|X\right|$ is paracompact, by $G_i$-equivariantly shrinking the spaces $V_i$, we may assume that the associated open cover $\{\left|V_i/G_i\right|\subseteq\left|X\right|\}_i$ of coarse spaces is locally finite.
We fix a covering sieve $S$ on $U\times_XU$ which is invariant under the `exchange' (i.e.\ `inverse') involution of $U\times_XU$ and is a connection sieve for both projections $U\times_XU\righttwoarrows U$.
We also fix a partition of unity $\{f_i:\left|X\right|\to\RR_{\geq 0}\}_i$ subordinate to the open cover $\{\left|V_i/G_i\right|\subseteq\left|X\right|\}_i$

Denote by $o(x)\subseteq U$ and $o_i(x)\subseteq V_i$ the fibers over $x\in\left|X\right|$, so $o(x)=\bigsqcup_io_i(x)$; similarly define $m(x)\subseteq U\times_XU$ and $m_{ij}(x)\subseteq V_i\times_XV_j$ with $m(x)=\bigsqcup_{i,j}m_{ij}(x)$.
These sets are finite since the $G_i$ are finite and $\{\left|V_i/G_i\right|\subseteq\left|X\right|\}_i$ is locally finite.
Furthermore, they have the discrete topology, since $U$ is Hausdorff and $U\times_XU$ is Hausdorff (since $U$ is Hausdorff and $X\to X\times X$ is separated).

The Hausdorff property implies that the inclusions $o_i(x)\subseteq V_i$ and $m_{ij}(x)\subseteq V_i\times_XV_j$ admit retractions defined in some open neighborhood.
Now the inverse images of small open neighborhoods $x\in\left|Z_x\right|\subseteq\left|X\right|$ (i.e.\ open substacks $Z_x\subseteq X$) form a basis of neighborhoods of $o_i(x)$ and $m_{ij}(x)$, so for sufficiently small $Z_x$, these inverse images are naturally disjoint unions $U_x=\bigsqcup_{o\in o(x)}U_o$ and $U_x\times_XU_x=\bigsqcup_{m\in m(x)}U_m$.
Note that this applies only inside $V_i$ and $V_i\times_XV_j$ for which $o_i(x)$ and $m_{ij}(x)$ are non-empty: the full inverse image of $Z_x$ inside $U$ may intersect other $V_i$ nontrivially.
By shrinking $Z_x$ further, we may ensure that the retraction $(U_x\times_XU_x\righttwoarrows U_x)\to(m(x)\righttwoarrows o(x))$ is a map of groupoids.
For later purposes, let us also take $Z_x$ small enough so that:
\begin{itemize}
\item If $x\in\left|V_i/G_i\right|$ then $\left|Z_x\right|\subseteq\left|V_i/G_i\right|$.
\item If $x\notin\supp f_i$ then $\left|Z_x\right|\cap\supp f_i=\varnothing$.
\item If $f_i(x)>0$ then $f_i>0$ over all of $\left|Z_x\right|$.
\end{itemize}
Each of these conditions can be ensured on its own, and since the open cover $\{\left|V_i/G_i\right|\subseteq\left|X\right|\}_i$ is locally finite, we can ensure all at once.

We also shrink $Z_x$ so as to ensure that each $U_m\in S$ (our chosen connection sieve), which has the following implication: given $o,o'\in U$ with $U_o\cap U_{o'}\ne\varnothing$, the relation $U_m\cap U_{m'}\ne\varnothing$ is a partial bijection between lifts $m,m'\in U\times_XU$ of $o$ and $o'$.
Furthermore, the domain of this bijection is as large as possible: for $U_o\cap U_{o'}\ne\varnothing$ (so $o\in o_i(x)$ and $o'\in o_i(x')$ for some $i$) with $x,x'\in\left|V_j/G_j\right|$, we get a full bijection between the inverse images of $o$ and $o'$ inside $m_{ij}(x)$ and $m_{ij}(x')$ (this follows since the projection $V_i\times_XV_j\to V_i$ is a finite covering space of degree $\left|G_j\right|$ over $U_o\cup U_{o'}$).

We now consider the nerves $N(U,\{U_o\}_{o\in U})$ and $N(U\times_XU,\{U_m\}_{m\in U\times_XU})$.
Note that for simplices in these nerves, namely subsets $O\subseteq U$ or $M\subseteq U\times_XU$ with $\bigcap_{o\in O}U_o\ne\varnothing$ or $\bigcap_{m\in M}U_m\ne\varnothing$, the maps $O\to\left|X\right|$ or $M\to\left|X\right|$ are injective.
The natural maps on index sets $U\times_XU\righttwoarrows U$ determine maps of nerves
\begin{equation*}
N(U\times_XU,\{U_m\}_{m\in U\times_XU})\righttwoarrows N(U,\{U_o\}_{o\in U}).
\end{equation*}
These maps are locally injective; indeed, local injectivity means that for every $o,o'\in U$ with $U_o\cap U_{o'}\ne\varnothing$ and every $m\in U\times_XU$ projecting to $o$, there is at most one lift $m'\in U\times_XU$ of $o'$ with $U_m\cap U_{m'}\ne\varnothing$, and this is a direct consequence of our assumption that every $U_m\in S$.
Let us now argue that there is a natural composition map
\begin{multline*}
N(U\times_XU,\{U_m\}_{m\in U\times_XU})\times_{N(U,\{U_o\}_{o\in U})}N(U\times_XU,\{U_m\}_{m\in U\times_XU})\\\longrightarrow N(U\times_XU,\{U_m\}_{m\in U\times_XU}).
\end{multline*}
More precisely, we claim that for non-empty finite subsets $M\subseteq U\times_XU$ with $\bigcap_{m\in M}U_m\ne\varnothing$ and $M'\subseteq U\times_XU$ with $\bigcap_{m\in M'}U_m\ne\varnothing$ projecting to $O\subseteq U$ (under the first and second projections, respectively), the subset $M''\subseteq U\times_XU$ defined by applying the composition map $(U\times_XU)\times_U(U\times_XU)\to U\times_XU$ to $M$ and $M'$ also satisfies $\bigcap_{m\in M''}U_m\ne\varnothing$.
This claim follows from the property that every $U_m\in S$ (indeed, this property implies that $\bigcap_{m\in M}U_m\xrightarrow\sim\bigcap_{o\in O}U_o\xleftarrow\sim\bigcap_{m\in M'}U_m$).
We have thus defined a locally injective simplicial complex groupoid
\begin{equation*}
M:=N(U\times_XU,\{U_m\}_{m\in U\times_XU})\righttwoarrows N(U,\{U_o\}_{o\in U})=:O.
\end{equation*}
We now show that this locally injective simplicial complex groupoid is sufficiently \'etale.
Every isomorphism class of vertex (equivalently, every $x\in\left|X\right|$ with $\left|Z_x\right|\ne\varnothing$) has a representative $o\in V_i\subseteq U$ with $f_i(o)=f_i(x)>0$.
To show that these representatives are \'etale, let $m\in U\times_XU$ be a morphism with source $o$ and target $o'$, and let $U_{o'}\cap U_{o''}\ne\varnothing$.
We must show that there is a bijection between morphisms $o\to o'$ and $o\to o''$.
The morphisms in question all have source inside $V_i$, so we really can consider just $V_i\times U$ for the present purpose.
Now the bulleted conditions on $\left|Z_x\right|$ from above imply that since $U_{o'}\cap U_{o''}\ne\varnothing$ and $f_i(o')>0$, we have $U_{o'}\cup U_{o''}\subseteq\left|V_i/G_i\right|\subseteq\left|X\right|$.
Thus over $U_{o'}\cap U_{o''}$ the map $V_i\times_XU\to U$ is a finite covering space of degree $\left|G_i\right|$.
Thus all points have the same number of lifts, so it follows that the connection sieve property gives us a bijection between lifts.

We have already seen above that our sufficiently \'etale locally injective simplicial complex groupoid satisfies the first hypothesis of Lemma \ref{MOgroup}, and the second hypothesis follows from the partial bijection property derived above from the connection sieve.
Thus by Lemma \ref{MOgroup}, there is a simplicial complex of groups $(Z,G)$ giving rise to the same simplicial groupoid $b_\bullet(Z,G)=[b_\bullet M\righttwoarrows b_\bullet O]$ and thus (since $M\righttwoarrows O$ is sufficiently \'etale) to the same geometric realization $\left\|(Z,G)\right\|=[\left\|M\right\|\righttwoarrows\left\|O\right\|]$.
The groups $G_z$ associated to vertices $z\in Z$ are by definition isotropy groups of the vertex groupoid $V(M)\righttwoarrows V(O)$, which are by definition isotropy groups of points of $X$.

To conclude, it remains to define a map $X\to\left\|(Z,G)\right\|$ which is injective on isotropy groups.
To define this map, we shrink the $Z_x$ so that the open cover $\{\left|Z_x\right|\subseteq\left|X\right|\}_{x\in\left|X\right|}$ is locally finite, and choose a partition of unity $\{g_x:\left|X\right|\to\RR_{\geq 0}\}_{x\in\left|X\right|}$ subordinate to the open cover $\{\left|Z_x\right|\subseteq\left|X\right|\}_{x\in\left|X\right|}$.
These maps $g_x$ lift to maps $g_o:U\to\RR_{\geq 0}$ supported inside $U_o$ and $g_m:U\times_XU\to\RR_{\geq 0}$ supported inside $U_m$.
The collection of these lifts defines a map of topological groupoids
\begin{equation*}
(U^0\times_XU^0\righttwoarrows U^0)\to\bigl(\left\|N(U\times_XU,\{U_m\}_{m\in U\times_XU})\right\|\righttwoarrows\left\|N(U,\{U_o\}_{o\in U})\right\|\bigr)
\end{equation*}
where $U^0$ is the disjoint union of the open loci $V_i^0\subseteq V_i$ where $f_i>0$ (the maps $g_o$ and $g_m$ do not define a map over all of $U\times_XU\righttwoarrows U$ due to the fact that $U_x=\bigsqcup_{o\in o(x)}U_o$ and $U_x\times_XU_x=\bigsqcup_{m\in m(x)}U_m$ may not be the full inverse images of $Z_x$ inside $U$ and $U\times_XU$).
It remains to check that this map is injective on isotropy groups; in other words, for $o,o'\in U^0$ we must show that the map $\{o\}\times_X\{o'\}\to\left\|N(U\times_XU,\{U_m\}_{m\in U\times_XU})\right\|$ is injective.
Distinct elements of $\{o\}\times_X\{o'\}$ are, in particular, distinct lifts of $o$, which therefore cannot lie in any common $U_m$ since $U_m\to U$ is injective, so we see that the map is indeed injective on isotropy groups.
\end{proof}

\begin{proof}[Proof of Theorem \ref{maingeneral}]
Apply Proposition \ref{maptosimplicial} to find a simplicial complex of finite groups $(Z,G)$ and a map $X\to\left\|(Z,G)\right\|$ which is injective on isotropy groups.
Now Theorem \ref{mainsimplicial} applies to $\left\|(Z,G)\right\|$ to give a vector bundle $V\to\left\|(Z,G)\right\|$, and the pullback of this bundle to $X$ satisfies the desired property since $X\to\left\|(Z,G)\right\|$ is injective on isotropy groups.
\end{proof}

\section{From vector bundles to principal bundles}\label{vectorprincipal}

We derive Corollaries \ref{maingeneralprincipal} and \ref{orbifoldpresentation} from Theorem \ref{maingeneral}.

A hermitian inner product on a complex vector bundle $V\to X$ is a map $h:V\times_XV\to\CC$ satisfying $h(v,w)=\overline{h(w,v)}$, $h(v,\alpha w)=\alpha h(v,w)$ for $\alpha\in\CC$, and $h(v,v)>0$ for $v\ne 0$ (meaning, these conditions are imposed on the fiber $h_x:V_x\times V_x\to\CC$ over each map $*\xrightarrow xX$).

\begin{lemma}\label{paracompactmetric}
A complex vector bundle over a paracompact orbispace $X$ admits a hermitian inner product.
\end{lemma}

\begin{proof}
Begin with an \'etale atlas $\bigsqcup_iU_i\to X$ such that the pullback of $V$ to every $U_i$ is trivial.
By Corollary \ref{atlasrefine}, we may refine this cover further so that each map $U_i\to X$ is the composition of $U_i\to U_i/G_i$ with an open embedding $U_i/G_i\hookrightarrow X$.
The pullback of $V$ to each $U_i$ is trivial, hence admits a hermitian inner product; by averaging, we may make it $G_i$-invariant thus giving a hermitian inner product $h_i$ on the restriction of $V$ to each $U_i/G_i\subseteq X$.
Since $\left|X\right|$ is paracompact, there is a partition of unity $\varphi_i$ subordinate to this open cover of $\left|X\right|$, and hence $\sum_i\varphi_ih_i$ is the desired hermitian inner product on $V$.
\end{proof}

\begin{corollary}\label{paracompactsplit}
Every short exact sequence of vector bundles over a paracompact orbispace splits.
\qed
\end{corollary}

Given a rank $n$ complex vector bundle $V\to X$ with hermitian inner product, the associated frame bundle $P\to X$ is defined by declaring that a map $U\to P$ ($U$ a topological space) is a map $U\to X$ together with an isomorphism $\CC^n\times U\to V\times_XU$ under which the pullback of $h$ is the standard hermitian inner product on $\CC^n$.
There is an action of $U(n)$ on $P$ (by precomposition with automorphisms of $\CC^n$ respecting its hermitian inner product), giving $P$ the structure of a principal bundle over $X$, meaning that for every map $U\to X$ from a topological space, there is an open cover $\{U_i\subseteq U\}_i$ such that $P\times_XU_i\to U_i$ is isomorphic to $U(n)\times U_i\to U_i$ with $U(n)$ acting by left multiplication on the first factor.

\begin{proof}[Proof of Corollary \ref{maingeneralprincipal}]
Let $V\to X$ be the rank $n$ complex vector bundle produced by Theorem \ref{maingeneral}.
By Lemma \ref{paracompactmetric}, there exists a hermitian inner product on $V$.
The total space $P$ of the associated principal bundle $P\to X$ has trivial isotropy since the isotropy groups of $X$ act faithfully on the fibers of $V$.

We claim that $P$ is a Hausdorff topological space.
By Corollary \ref{criterionrepresentability}, it is enough to show that $P$ is an orbispace.
Since $P\to X$ is representable, the pullback of an \'etale atlas for $X$ is an \'etale atlas for $P$.
The diagonal of $P$ may be expressed as the composition $P\to P\times_XP\to P\times P$.
The second map $P\times_XP\to P\times P$ is proper, being a pullback of the diagonal of $X$.
To check that the first map $P\to P\times_XP$ is proper, it suffices by descent to show that its pullback under an \'etale atlas $U\to X$ is proper.
This pullback is proper since $P\times_XU\to U$ is a principal $U(n)$ bundle (of topological spaces).
We have thus shown that $P$ is a topological space.

The map $P\to X$ is representable and admits local sections (by definition), so the topological groupoid $P\times_XP\righttwoarrows P$ presents the stack $X$.
It can be seen by inspection that $P\times_XP\righttwoarrows P$ is the action groupoid of the $U(n)$ action on $P$.
\end{proof}

Let $\Sm$ denote the category of topological spaces equipped with a maximal atlas of charts from open sets of $\bigsqcup_{n\geq 0}\RR^n$ with smooth transition maps (a smooth manifold an object of $\Sm$ whose underlying topological space is Hausdorff).
A smooth structure on a stack $F:\Top^\op\to\Grpd$ is a substack $F^\sm$ of the pullback of $F$ under the forgetful functor $\Sm\to\Top$ (meaning that $F^\sm(U)$ is a full subcategory of $F(U)$ for $U\in\Sm$); maps $U\to F$ lying in (the essential image of) $F^\sm$ are then called smooth.
Given a map of stacks $F\to G$ and a smooth structure on $G$, we may consider the pullback smooth structure on $F$, defined as those maps $U\to F$ whose composition with $F\to G$ is smooth.
A map of stacks equipped with smooth structures $(F,F^\sm)\to(G,G^\sm)$ is called smooth iff the composition of any smooth map $U\to F$ with $F\to G$ is a smooth map $U\to G$ (equivalently, $F^\sm$ is contained inside the pullback of $G^\sm$).
Stacks with smooth structures form a 2-category just like stacks.
This category is complete, and limits $\lim_\alpha(F_\alpha,F_\alpha^\sm)$ are calculated by taking the limit of the underlying stacks $\lim_\alpha F_\alpha$ and declaring a map $U\to\lim_\alpha F_\alpha$ to be smooth iff every induced map $U\to F_\alpha$ is smooth.
The category $\Sm$ embeds fully faithfully into the category of stacks with smooth structures, by sending $X\in\Sm$ to the Yoneda functor of its underlying topological space, equipped with the smooth structure consisting of those maps $U\to X$ which are morphisms in $\Sm$.

\begin{definition}
A \emph{smooth orbifold} $X$ is an orbispace equipped with a smooth structure such that for every (equivalently, some) \'etale atlas $U\to X$, we have $U\in\Sm$ when $U$ is equipped with the pullback of the smooth structure on $X$.
\end{definition}

By Proposition \ref{orbispacelocal}, a stack $X$ with a smooth structure is a smooth orbifold iff $\left|X\right|$ is Hausdorff and there is an open cover of $X$ by $V_i/G_i$ for smooth manifolds $V_i$ equipped with smooth actions of finite groups $G_i$.

\begin{proof}[Proof of Corollary \ref{orbifoldpresentation}]
Let $V\to X$ be the complex vector bundle produced by Theorem \ref{maingeneral}.
It suffices to define a smooth structure on $V$, because then we may follow the arguments of Lemma \ref{paracompactmetric} and Corollary \ref{maingeneralprincipal} in the smooth category.

We begin by arguing that a complex vector bundle $V$ over a smooth orbifold $X$ has a smooth structure in a neighborhood of any given point of $\left|X\right|$.
We may thus assume that $X=U/G$ for a smooth manifold $U$ acted on smoothly by a finite group $G$, and that the pullback of $V$ to $U$ is trivial.
By shrinking this chart further (and adjusting $G$), we may assume that our given point of $\left|X\right|$ corresponds to a point $u\in U$ fixed by $G$.
Choose any trivialization $V_u\times U\to V\times_XU$ where $V_u$ denotes the fiber over $u$.
By averaging, we may make this map $G$ equivariant, and it remains an isomorphism in a neighborhood of $u$.
We thus obtain a trivialization of $V\times_XU$ near $u$ in which the action of $G$ is constant (independent of the $U$ coordinate).
The standard smooth structure in this trivialization is thus, in particular, invariant under the action of $G$, giving the desired smooth structure on $V$ near our given point of $\left|X\right|$.

Having shown the existence of smooth structures locally, we now show how to patch them together.
To show that there exists a smooth structure on $V$ over all of $X$, it suffices to show that for open subsets $A,B\subseteq\left|X\right|$ and smooth structures on $V|_A$ and $V|_B$, there exists a smooth structure on $V|_{A\cup B}$ restricting to the given smooth structure on $V|_A$ (indeed, this allows us to patch together smooth structures over arbitrary unions of open sets, by choosing a well ordering and adding one open set at a time).
To prove this pairwise patching statement, it suffices to show that an isomorphism of smooth vector bundles may be approximated by a smooth isomorphism (then apply this to the identity map on $V|_{A\cap B}$ equipped with the restrictions of the two given smooth structures).
Since $\left|X\right|$ is locally the quotient of Euclidean space by a finite group action, it is locally metrizable; since it is paracompact and Hausdorff, it is thus metrizable.
Thus every open subset of $\left|X\right|$ is metrizable, hence paracompact.
We may now conclude with a smooth partition of unity argument.
\end{proof}

\section{\texorpdfstring{$K$-theory}{K-theory} of orbispaces}\label{ktheory}

We derive Corollary \ref{enoughvectorbundles} from Theorem \ref{maingeneral}.
We then derive from Corollary \ref{enoughvectorbundles} some basic properties of the $K$-theory of finite rank vector bundles on orbispaces (some of this derivation is borrowed, with various technical differences, from \cite[\S 3]{luckoliver} and \cite{segalequivariant}).
In what follows, `vector bundle' can be taken to mean either real vector bundle or complex vector bundle (always of finite rank).

\begin{proof}[Proof of Corollary \ref{enoughvectorbundles}]
Denote our given map by $f:X\to Y$, and let $E$ be the given vector bundle on $X$.
Let $F$ be any vector bundle on $Y$ satisfying the conclusion of Theorem \ref{maingeneral}.

We first consider the local situation on $X$.
Fix $x\in X$, and let $U/G\hookrightarrow X$ be an open embedding sending $u\in U$ fixed by $G$ to $x\in X$ and inducing an isomorphism $G\xrightarrow\sim G_x$ (such an open embedding exists by the proof of Proposition \ref{orbispacelocal}).
By shrinking $U$, we may ensure that the pullbacks of $E$ and $F$ to $U$ are trivial.
Choose any map on $U$ from the pullback of $E$ to the pullback of $F^{\oplus N}$ (some integer $N<\infty$) which is injective and $G$-equivariant at $u$, and average it to make it equivariant everywhere.
This produces over $U/G$ a map from $E$ to $f^*F^{\oplus N}$ which is injective at $x$.
Now since $\left|X\right|$ is paracompact Hausdorff, hence normal, there exists a continuous function $\varphi:\left|X\right|\to[0,1]$ supported inside $\left|U/G\right|$ with $\varphi(x)=1$.
Multiplying by this function yields a map $E\to f^*F^{\oplus N}$ defined on all of $X$ which is injective at $x$.
Note that the integer $N$ may be taken independent of $x$ since the rank of $E$ is bounded.

We now combine the above maps to give an everywhere injective map $E\to f^*F^{\oplus M}$ as follows.
Fix an open cover $\left|X\right|=\bigcup_iU_i$ and maps $f_i:E\to f^*F^{\oplus N}$ defined over all of $X$ which are injective over $U_i$.
Since $\left|X\right|$ is coarsely finite-dimensional, we may refine this open cover so that it is locally finite and has nerve of dimension $\leq d$ for some integer $d<\infty$.
As in the proof of Lemma \ref{connectionparacompact}, there is yet another open cover of $\left|X\right|$ by open sets $V_I$ indexed by the non-empty subsets $I$ of the index set of the $U_i$, such that $V_I\subseteq\bigcap_{i\in I}U_i$ and $V_I\cap V_J=\varnothing$ unless $I\subseteq J$ or $J\subseteq I$.
Note that these conditions imply that $V_I=\varnothing$ unless $\left|I\right|\leq d+1$ and that $V_I\cap V_J=\varnothing$ if $\left|I\right|=\left|J\right|$ and $I\ne J$.
Now let $\sum_I\varphi_I\equiv 1$ be a partition of unity subordinate to the open cover $\left|X\right|=\bigcup_IV_I$.
We may now define our desired everywhere injective map $E\to f^*F^{\oplus(d+1)N}$ (so $M=(d+1)N$) by the formula $\sum_I\alpha_{\left|I\right|}\varphi_If_{i(I)}$, where $i(I)$ is any choice of index in the set $I$, and $\alpha_r:F^{\oplus N}\hookrightarrow(F^{\oplus N})^{\oplus(d+1)}=F^{\oplus(d+1)N}$ is the inclusion of the $r$th direct summand.

We have thus constructed an injection $E\to f^*F^{\oplus M}$ over $X$.
\end{proof}

\begin{definition}
For any stack $X$, denote by $\Vect(X)$ the set of isomorphism classes of vector bundles on $X$ \emph{of bounded rank}.
\end{definition}

Direct sum of vector bundles equips $\Vect(X)$ with the structure of an abelian monoid.
A map of stacks $X\to Y$ induces a pullback map $\Vect(Y)\to\Vect(X)$.

The reason we restrict attention to vector bundles of bounded rank is so that we may appeal to Corollary \ref{enoughvectorbundles}.
A counterexample of L\"uck--Oliver \cite[Example 3.11]{luckoliver} shows that the $K$-theory of vector bundles of unbounded rank fails to be a cohomology theory on (possibly infinite) simplicial complexes of groups, whereas we shall see below that the $K$-theory of vector bundles of bounded rank is a cohomology theory for such.

For an abelian monoid $M$, its group completion $M[-M]$ is the quotient of the free abelian group on the underlying set of $M$ by the subgroup generated by $[a]+[b]-[a+b]$ for $a,b\in M$.
The map $M\to M[-M]$ is universal (initial) among maps from $M$ to an abelian group.

\begin{definition}
$K^0(X)$ is the group completion of $\Vect(X)$ for any stack $X$.
\end{definition}

A map of stacks $X\to Y$ induces a pullback map $K^0(Y)\to K^0(X)$.
The functor $K^0$ is finitely additive, in the sense that $K^0(\varnothing)=0$ and the natural map $K^0(X\sqcup Y)\to K^0(X)\oplus K^0(Y)$ is an isomorphism.
In the present context of vector bundles of bounded rank, additivity does not hold for general infinite disjoint unions.

Eventually, we will restrict attention to the $K$-theory of orbispaces satisfying the hypothesis of Theorem \ref{maingeneral}.
However, we will impose such hypotheses gradually, as they become relevant.

To show that $K$-theory is homotopy invariant, the following is the key assertion.

\begin{lemma}\label{intervalpullback}
For any paracompact orbispace $X$, every vector bundle on $X\times[0,1]$ is pulled back from $X$.
\end{lemma}

\begin{proof}
Let a vector bundle $V$ over $X\times[0,1]$ be given.

We first discuss the local structure around a given point $x\in X$.
Fix an open embedding $Y/G\hookrightarrow X$ and a lift $y\in Y$ of $x$ fixed by $G$.
We consider the pullback bundle $V|_{Y\times[0,1]}$ on $Y\times[0,1]$, which is a $G$-equivariant vector bundle.
For each $t\in[0,1]$, there exist open neighborhoods $t\in T\subseteq[0,1]$ and $y\in U\subseteq Y$ such that $V|_{U\times T}$ is trivial.
By compactness of $[0,1]$, we may assume $U$ is independent of $t$.
Replacing $Y$ with $U$, we may assume that for each $t\in[0,1]$ there exists an open neighborhood $t\in T\subseteq[0,1]$ such that $V|_{Y\times T}$ is trivial.
A trivialization induces a map $V|_{Y\times T}\to p_T^*V|_{\{y\}\times T}$ which is the identity over $\{y\}\times T$.
Averaging makes this map $G$-equivariant, so it descends to a map $V|_{Y/G\times T}\to p_T^*V|_{\{y\}/G\times T}$ which is still the identity over $\{y\}\times T$.
It is thus an isomorphism over some $U\times T'$; another compactness argument and shrinking of $Y$ ensures that we have a collection of isomorphisms $V|_{Y/G\times T}\to p_T^*V|_{\{y\}/G\times T}$ which are the identity over $\{y\}\times T$.
Patching these together via a partition of unity on $[0,1]$ and further shrinking $Y$ produces an isomorphism
\begin{equation*}
V|_{Y/G\times[0,1]}\xrightarrow\sim p_{[0,1]}^*V|_{\{y\}/G\times[0,1]}
\end{equation*}
which is the identity over $\{y\}\times[0,1]$.
Any vector bundle over $\BB G\times[0,1]$ is pulled back from $\BB G$, and hence we conclude that $V|_{Y/G\times[0,1]}$ is pulled back from $Y/G$, for some neighborhood $Y/G$ of $x\in X$.

We now globalize.
Begin with an open cover $\left|X\right|=\bigcup_iU_i$ and over each $U_i$ an isomorphism $\xi_i:V|_{U_i\times[0,1]}\to(p_X^*V|_{X\times\{0\}})|_{U_i\times[0,1]}$ which is the identity over $U_i\times\{0\}$.
Note that for any $x\in U_i$ and any $t,t'\in[0,1]$, the map $\xi_i$ determines an isomorphism
\begin{equation*}
V_{(x,t)}\xrightarrow{\xi_i^{-1}\circ\xi_i}V_{(x,t')}
\end{equation*}
which is the specialization of an isomorphism between the pullbacks of $V$ under the two projections $X\times[0,1]^2\to X\times[0,1]$.
Using this observation, we may now patch together the $\xi_i$ into an isomorphism $V\to p_X^*V|_{X\times\{0\}}$ as follows.
Fix a partition of unity $\varphi_i:\left|X\right|\to\RR_{\geq 0}$ subordinate to the open cover $\left|X\right|=\bigcup_iU_i$, and fix an arbitrary total ordering of the index set of the open cover.
For $(x,t)\in X\times[0,1]$, let $i=1,\ldots,k$ denote the indices of the open sets $U_i$ containing $x$, ordered as in the fixed total order, and consider the composition of isomorphisms
\begin{equation*}
V_{(x,t)}\xrightarrow{\xi_1^{-1}\circ\xi_1}V_{(x,(1-\varphi_1(x))\cdot t)}\xrightarrow{\xi_2^{-1}\circ\xi_2}V_{(x,(1-\varphi_1(x)-\varphi_2(x))\cdot t)}\xrightarrow{\xi_3^{-1}\circ\xi_3}\cdots\xrightarrow{\xi_k^{-1}\circ\xi_k}V_{(x,0)}.
\end{equation*}
This fiberwise description now translates into the desired isomorphism of vector bundles $V\to p_X^*|_{X\times\{0\}}$.
\end{proof}

Two maps $X\to Y$ will be called homotopic iff there exists a map $X\times[0,1]\to Y$ whose restrictions to $X\times\{0\}$ and $X\times\{1\}$ are the two given maps.
A map $X\to Y$ is called a homotopy equivalence iff there exists a map $Y\to X$ such that the compositions $X\to Y\to X$ and $Y\to X\to Y$ are homotopic to the respective identity maps.

\begin{corollary}[Homotopy Invariance]\label{htpyabsolute}
If $X$ is a paracompact orbispace, then homotopic maps $X\to Y$ induce the same map $K^0(Y)\to K^0(X)$.
\qed
\end{corollary}

For any map of abelian monoids $M'\to M$, the quotient $M/M'$ is, as a set, the quotient of $M$ by the equivalence relation $x\sim_{M'}y$ iff there exist $a,b\in M'$ with $x+a=y+b$, equipped with the descent of the monoid operation from $M$.
The map $M\to M/M'$ is initial among maps from $M$ to an abelian monoid sending $M'$ to zero.
If (the image inside $M$ of) $M'$ is an abelian group, then the kernel of the quotient map $M\to M/M'$ is precisely $M'$.
The group completion $M[-M]$ is the quotient of $M\times M$ by the diagonal submonoid $M$.

\begin{definition}
For any map of stacks $f:Y\to X$, the relative $K$-theory $K^0(X,Y)$ is defined as follows.
We consider the set $\Tri(X,Y)$ of isomorphism classes of triples $(E_0,E_1,i)$ where $E_0,E_1$ are vector bundles on $X$ of bounded rank and $i:f^*E_0\to f^*E_1$ is an isomorphism between their pullbacks to $Y$.
Now $\Tri(X,Y)$ is an abelian monoid under direct sum, and it contains $\Vect(X)$ as the submonoid of triples of the form $(E,E,f^*\id_E)$.
The relative $K$-theory is the quotient
\begin{equation*}
K^0(X,Y):=\Tri(X,Y)/\Vect(X).
\end{equation*}
Note that \emph{a priori} $K^0(X,Y)$ is merely an abelian monoid, not an abelian group.
\end{definition}

A map $(Y\to X)\to(Y'\to X')$ (i.e.\ a commutative square) induces a map $K^0(X',Y')\to K^0(X,Y)$.
There is a natural map $K^0(X,Y)\to K^0(X)$ given by sending $(E_0,E_1,i)$ to the formal difference $[E_0]-[E_1]$, and the composition $K^0(X,Y)\to K^0(X)\to K^0(Y)$ is zero.
The map $K^0(X,\varnothing)\to K^0(X)$ is an isomorphism (since $\Tri(X,\varnothing)$ is simply $\Vect(X)\times\Vect(X)$).

A \emph{pair} $(X,A)$ shall mean that $A$ is a closed substack of $X$.
A map of pairs $(X,A)\to(Y,B)$ means a map $X\to Y$ sending $A$ inside $B$.
Note that paracompactness, coarse finite-dimensionality, being an orbispace, and the hypothesis of Theorem \ref{maingeneral} all pass to closed substacks.

\begin{remark}
Even for very nice orbispace pairs $(X,A)$, we \emph{cannot} in general form the quotient $X/A$ in a reasonable way.
For example, for an orbispace $X$, if we could reasonably define $(X\times[0,1])/(X\times\{0\})$ as an orbispace, it would follow that $X$ is (globally) the quotient of a topological space by a finite group action (a very special property).
\end{remark}

For any pair $(X,A)$, let $\cyl(X,A)$ denote the pair $((X\times\{0\})\cup(A\times[0,1]),A\times\{1\})$.
Here $(X\times\{0\})\cup(A\times[0,1])$ denotes the union of closed substacks of $X\times[0,1]$ (recall that closed substacks of a stack $Z$ are in bijective correspondence with closed subsets of $\left|Z\right|$, so by union of closed substacks we mean union of subsets of $\left|Z\right|$).
There is a natural map of pairs $\cyl(X,A)\to(X,A)$.
A map of pairs $f:(X,A)\to(X',A')$ induces a map $\cyl(f):\cyl(X,A)\to\cyl(X',A')$.

\begin{lemma}\label{relabgrp}
For any pair $(X,A)$ where $A$ is a paracompact orbispace, $K^0(\cyl(X,A))$ is an abelian group.
\end{lemma}

\begin{proof}
Let $(E_0,E_1,i:E_0|_{A\times\{1\}}\to E_1|_{A\times\{1\}})$ be a triple representing an arbitrary element of $K^0(\cyl(X,A))$.
We claim that the triple $(E_1,E_0,-i)$ is an inverse to it.
It suffices to show that
\begin{equation*}
(E_0\oplus E_1,E_0\oplus E_1,(\begin{smallmatrix}0&-i\\i&\hfill 0\end{smallmatrix}))\quad\text{and}\quad(E_0\oplus E_1,E_0\oplus E_1,(\begin{smallmatrix}1&0\\0&1\end{smallmatrix}))
\end{equation*}
are isomorphic.
In other words, it suffices to show that $(\begin{smallmatrix}0&-i\\i&\hfill 0\end{smallmatrix})$ is the restriction to $A\times\{1\}$ of an automorphism of $E_0\oplus E_1$.
By Lemma \ref{intervalpullback}, there exist isomorphisms $E_i|_{A\times[0,1]}=p_A^*(E_i|_{A\times\{1\}})$.
In such coordinates, the desired automorphism of $E_0\oplus E_1$ may be given by
\begin{equation*}
\left(\begin{matrix}\hfill\cos\frac\pi 2t&-i\sin\frac\pi 2t\\i\sin\frac\pi 2t&\hfill\cos\frac\pi 2t\end{matrix}\right)
\end{equation*}
on $A\times[0,1]$ and the identity on $X\times\{0\}$ (note that to specify a map of vector bundles over a stack, it suffices to specify its restriction to each member of a finite cover by closed substacks, subject to the requirement that these restrictions agree on their pairwise overlaps).
\end{proof}

\begin{lemma}\label{kzerocylsurjective}
For any pair $(X,A)$ where $A$ is a paracompact orbispace, the map $K^0(X,A)\to K^0(\cyl(X,A))$ is surjective.
\end{lemma}

\begin{proof}
We claim that every vector bundle on $(X\times\{0\})\cup(A\times[0,1])$ is pulled back from $X$.
Let $E$ be a vector bundle on $\cyl(X,A)$, let $F$ denote its restriction to $X=X\times\{0\}$, and let us show that $E=p_X^*F$.
The identity map is an identification of $E$ and $p_X^*F$ over $X\times\{0\}$, and by Lemma \ref{intervalpullback} there is an identification of $E$ and $p_X^*F$ over $A\times[0,1]$ which agrees with the identity over $A\times\{0\}$.
These isomorphisms $E\to p_X^*F$ thus patch together to define the desired isomorphism.

Since every vector bundle on $(X\times\{0\})\cup(A\times[0,1])$ is pulled back from $X$, it follows that $\Tri(X,A)\to\Tri(\cyl(X,A))$ is surjective, so we are done.
\end{proof}

\begin{lemma}\label{vborbispacelocal}
For every vector bundle $E$ over an orbispace $X$, there exists an open cover $X=\bigcup_iZ_i/G_i$ such that the restriction of $E$ to each $Z_i/G_i$ is pulled back from $*/G_i$ (equivalently, the pullback of $E$ to $Z_i$ is $G_i$-equivariantly trivial).
\end{lemma}

\begin{proof}
Every point $x\in X$ has an open neighborhood of the form $Z/G$ by Proposition \ref{orbispacelocal}.
Choose a lift $z\in Z$ of $x$, and by shrinking $Z$ and replacing $G$ with the stabilizer of $z$, assume that $G$ fixes $z$.
Choose a map from $E|_Z$ to the trivial bundle $E_z\times Z$ which is the identity at $z$, average this map to make it $G$-equivariant, and shrink $Z$ so that this map is an isomorphism over all of $Z$.
\end{proof}

\begin{lemma}\label{paracompactquotient}
Let $G\acts Z$ be a finite group action on a Hausdorff topological space.
If $Z/G$ is paracompact, then so is $Z$.
\end{lemma}

\begin{proof}
Let an open covering $Z=\bigcup_\alpha V_\alpha$ be given.
For every $z\in Z$, there exists an open neighborhood $U$ of $z$ which is $G_z$-invariant and whose $G/G_z$-translates are disjoint and each contained in some (possibly different) $V_\alpha$.
The images of such $U$ in $Z/G$ form an open covering of $Z/G$.
Refining this to a locally finite covering and taking inverse images in $Z$ gives the desired locally finite refinement of the original covering.
\end{proof}

\begin{lemma}\label{sectionextension}
Let $X$ be a paracompact orbispace, let $E$ be a vector bundle over $X$, and let $Y\subseteq X$ be a closed substack.
Every section of $E|_Y$ extends to a section of $E$.
\end{lemma}

\begin{proof}
Fix an open cover $X=\bigcup_iZ_i/G_i$ such that the pullback of $E$ to each $Z_i$ is trivialized $G_i$-equivariantly (Lemma \ref{vborbispacelocal}).
Let $\varphi_i:\left|X\right|\to[0,1]$ be a partition of unity subordinate to this covering.
Now $\supp\varphi_i$ is a closed subset of a paracompact Hausdorff space, hence paracompact Hausdorff; its inverse image inside $Z_i$ is thus paracompact Hausdorff by Lemma \ref{paracompactquotient} (recall $Z_i$ is Hausdorff).
Thus by the Tietze extension theorem, our given section on $Z_i\times_X(Y\cap\supp\varphi_i)$ extends to $Z_i\times_X\supp\varphi_i$.
We can make it $G_i$-equivariant by averaging, so our given section on $Y\cap\supp\varphi_i$ extends to $\supp\varphi_i$.
Now $(\supp\varphi_i)^\circ$ is an open cover of $X$, so pick another partition of unity $\psi_i$ subordinate to this cover, and use it to patch together the extended sections on each $\supp\varphi_i$.
\end{proof}

\begin{lemma}\label{kzerocyl}
For any paracompact orbispace pair $(X,A)$, the map $K^0(X,A)\to K^0(\cyl(X,A))$ is an isomorphism.
\end{lemma}

\begin{proof}
We know from the proof of Lemma \ref{kzerocylsurjective} that every vector bundle on $(X\times\{0\})\cup(A\times[0,1])$ is pulled back from $X$.
Given this fact, it suffices to show that the pullback map $\Tri(X,A)\to\Tri(\cyl(X,A))$ is injective.
Thus suppose that we are given two triples $(E_0,E_1,i:E_0|_A\to E_1|_A)$ and $(E_0',E_1',i':E_0'|_A\to E_1'|_A)$ whose pullbacks to $\cyl(X,A)$ coincide, meaning that there are isomorphisms $\alpha_i:p_X^*E_i\to p_X^*E_i'$ on $(X\times\{0\})\cup(A\times[0,1])$ intertwining $i$ and $i'$ over $A\times\{1\}$.
Now these maps $\alpha_i$ extend to all of $X\times[0,1]$ by Lemma \ref{sectionextension}, and these extended maps are isomorphisms over open neighborhood of $(X\times\{0\})\cup(A\times[0,1])$.
Since $X\times[0,1]\to X$ is universally closed (being a pullback of $[0,1]\to *$), this neighborhood contains $U\times[0,1]$ for some open neighborhood $U$ of $A\subseteq X$.
Now a paracompact Hausdorff space is normal, so by Urysohn's lemma there exists a continuous function $\varphi:\left|X\right|\to[0,1]$ supported inside $U$ which is identically $1$ on $A$.
Pulling back $\alpha_i$ under the graph of $\varphi$ defines isomorphisms $E_i\to E_i'$ whose restrictions to $A$ intertwine $i$ and $i'$, showing that the original triples on $(X,A)$ are isomorphic.
\end{proof}

\begin{proposition}[Relative Homotopy Invariance]\label{relhtpyinv}
For any map of paracompact orbispace pairs $(X,A)\to(X',A')$ whose constituent maps $X\to X'$ and $A\to A'$ are individually homotopy equivalences, the map $K^0(X',A')\to K^0(X,A)$ is an isomorphism.
\end{proposition}

\begin{proof}
We begin by showing that the two maps $K^0(X\times[0,1],A\times[0,1])\to K^0(X,A)$ (pullback under $\times\{0\}$ and $\times\{1\}$) coincide (from which it follows that homotopic maps of pairs induce the same map on relative $K^0$).
Since vector bundles on $X\times[0,1]$ are pulled back from $X$ by Proposition \ref{intervalpullback}, this amounts to showing that triples $(E_0,E_1,i),(E_0,E_1,j)\in\Tri(X,A)$ represent the same element of $K^0(X,A)$ if $i$ and $j$ are homotopic.
The construction from the proof of Lemma \ref{relabgrp} shows that the homotopy between $i^{-1}\circ j$ and the identity gives rise to an isomorphism between the pullbacks of these triples to $\cyl(X,A)$, which is enough by Lemma \ref{kzerocyl}.

Since homotopic maps of pairs induce the same map on relative $K^0$, it follows that homotopy equivalences of pairs induce isomorphisms on relative $K^0$.
In light of Lemma \ref{kzerocyl}, it thus suffices to show that a map of pairs $(X,A)\to(X',A')$ whose constituent maps $X\to X'$ and $A\to A'$ are individually homotopy equivalences induces a homotopy equivalence of pairs $\cyl(X,A)\to\cyl(X',A')$.
This is well known: given $(f,g):(X,A)\to(X',A')$, let $p:X'\to X$ and $q:A'\to A$ be homotopy inverses to $f$ and $g$.
To define a map $\cyl(X',A')\to\cyl(X,A)$ which is $p$ on $X'\times\{0\}$ and $q$ on $A'\times\{1\}$, we need a homotopy between the two maps $p|_{A'},q:A'\to X$.
A distinguished homotopy class of such homotopies is furnished by composing further with the homotopy equivalence $f:X\to X'$ and fixing homotopies between $f\circ p$ and $\id_{X'}$ and between $f|_A\circ q=g\circ q$ and $\id_{A'}$.
One then checks that this map $\cyl(X',A')\to\cyl(X,A)$ is a homotopy inverse to the map $\cyl(X,A)\to\cyl(X',A')$.
\end{proof}

\begin{lemma}\label{bundleextension}
Let $X$ be a paracompact orbispace.
Every vector bundle $E$ over a closed substack $Y\subseteq X$ is the restriction of a vector bundle over some open substack $U\subseteq X$ containing $Y$.
\end{lemma}

\begin{proof}
Fix a locally finite open cover $X=\bigcup_iZ_i/G_i$ such that the pullback of $E$ to each $(Z_i)_Y:=Z_i\times_XY\subseteq Z_i$ is trivialized $G_i$-equivariantly (Lemma \ref{vborbispacelocal}).
Such trivializations evidently extend $E|_{Y\cap(Z_i/G_i)}$ to $Z_i/G_i$.
We patch together these extensions $E_i$ on $Z_i/G_i$ as follows.

Choose closed substacks $K_i\subseteq X$ with $K_i\subseteq Z_i/G_i$ and whose interiors cover $X$.
We thus have a collection of transition functions $\alpha_{ij}:K_i\cap K_j\cap Y\to\Hom(E_i,E_j)$ satisfying $\alpha_{ii}=\id$ and $\alpha_{ij}\alpha_{jk}\alpha_{ki}=\id$ over $K_i\cap K_j\cap K_k\cap Y$.
We now execute the following operation for every pair $(i,j)$ inductively according to an arbitrary well-ordering of such pairs.
For $i=j$, do nothing.
For $i\ne j$, choose an extension of $\alpha_{ij}$ from $K_i\cap K_j\cap Y$ to $K_i\cap K_j$ using Lemma \ref{sectionextension}.
This extension remains an isomorphism in a neighborhood of $K_i\cap K_j\cap Y$.
Since (the coarse space of) $K_i\cap K_j$ is paracompact Hausdorff, hence normal, we may choose a closed substack $A\subseteq K_i\cap K_j$ whose interior contains $K_i\cap K_j\cap Y$ and over which the extension of $\alpha_{ij}$ is an isomorphism.
Now this new $\alpha_{ij}$ over $A$ gives unique extensions of the remaining $\alpha_{k\ell}$ to $K_k\cap K_\ell\cap(Y\cup A)$ such that the cocycle condition is satisfied.
We now replace $Y$ with $Y\cup A$ and go on to the next pair of indices.
Note that since our cover $X=\bigcup_iZ_i/G_i$ is locally finite, the set $Y$ remains closed even after possibly infinitely many steps.

After processing every pair $(i,j)$, our extended transition functions define a vector bundle on a closed substack $\bar Y\subseteq X$ whose restriction to $Y\subseteq\bar Y$ is $E$, and by construction $Y$ is contained in the interior of $\bar Y$.
\end{proof}

We now come to the exactness and excision properties of $K$-theory, where we will finally make use of Corollary \ref{enoughvectorbundles}.

For the proof of exactness, we will make use of the following notion.
Let us call a map of abelian monoids $f:M\to N$ \emph{cofinal} iff for every $n\in N$ there exist $m\in M$ and $n'\in N$ with $f(m)=n+n'$.
The quotient $M/M'$ is an abelian group iff $M'\to M$ is cofinal.
The conclusion of Corollary \ref{enoughvectorbundles} for a map $Y\to X$ is that the pullback map $\Vect(X)\to\Vect(Y)$ is cofinal (recalling from Corollary \ref{paracompactsplit} that every inclusion of vector bundles over a paracompact orbispace is split).

\begin{proposition}[Exactness]\label{lesktheory}
Let $X$ be an orbispace satisfying the hypothesis of Theorem \ref{maingeneral}, let $A\subseteq X\supseteq Y$ be closed substacks, let $B:=A\cap Y$ be their intersection, and let $A\cup_BY\subseteq X$ be their union.
The following sequence is exact:
\begin{equation*}
K^0(X,A\cup_BY)\to K^0(X,A)\to K^0(Y,B).
\end{equation*}
\end{proposition}

\begin{proof}
The following sequence is exact
\begin{equation*}
\Tri(X,A\cup_BY)\to\Tri(X,A)\to(\Tri(Y,B),\Vect(Y)),
\end{equation*}
in the sense that the image of the first map coincides with the inverse image of $\Vect(Y)$ under the second map.
Indeed, given a triple $(E_0,E_1,i)\in\Tri(X,A)$ for which $i|_B$ extends to $Y$, we may glue this extension to $i$ over $B$ to lift $(E_0,E_1,i)$ from $\Tri(X,A)$ to $\Tri(X,A\cup_BY)$.
Quotienting this sequence by $\Vect(X)$, we conclude that
\begin{equation*}
K^0(X,A\cup_BY)\to K^0(X,A)\to(\Tri(Y,B)/\Vect(X),\Vect(Y)/\Vect(X))
\end{equation*}
is exact.
Now $\Vect(X)\to\Vect(Y)$ is cofinal by Corollary \ref{enoughvectorbundles}, so $\Vect(Y)/\Vect(X)$ is an abelian group, so the kernel of the map from $\Tri(Y,B)/\Vect(X)$ to its quotient by $\Vect(Y)/\Vect(X)$ is precisely $\Vect(Y)/\Vect(X)$.
This quotient is simply $\Tri(Y,B)/\Vect(Y)=K^0(Y,B)$, so we are done.
\end{proof}

For the proof of excision, we will make use of the following alternative description of relative $K$-theory in terms of direct limits.
Let $\widehat\Vect(X)$ denote the category whose objects are vector bundles on $X$ of bounded rank and whose morphisms are homotopy classes of injective maps.
The category $\widehat\Vect(X)$ is \emph{filtered}, meaning that (1) it is non-empty, (2) for every pair of objects $x,y$, there exist morphisms $x\to z\leftarrow y$, and (3) for every pair of morphisms $x\righttwoarrows y$ there exists a morphism $y\to z$ such that the two compositions $x\to z$ coincide.
Indeed, (1) we have the zero vector bundle, (2) given vector bundles $E$ and $F$, they both admit a morphism to $E\oplus F$, and (3) given two injections $E\hookrightarrow F$, they become homotopic after composing with the inclusion $F\hookrightarrow F\oplus E$.

For any map of stacks $f:Y\to X$ and any vector bundle $F$ on $Y$, let $\Vect(X,F)$ denote the set of isomorphism classes of pairs $(V,i)$ consisting of a vector bundle $V$ on $X$ of bounded rank and an isomorphism $i:f^*V\xrightarrow\sim F$.
If $X$ is a paracompact orbispace, $\Vect(X,f^*E)$ forms a directed system over $E\in\widehat\Vect(X)$.
Indeed, given an inclusion $E\hookrightarrow E'$, there is a natural map $\Vect(X,f^*E)\xrightarrow{{}\oplus(E'/E)}\Vect(X,f^*E')$ since the inclusion $E\hookrightarrow E'$ has a(n automatically unique up to homotopy) splitting by Corollary \ref{paracompactsplit}.
For such $X$, we claim that $K^0(X,Y)$ may be expressed as the direct limit
\begin{equation*}
K^0(X,Y)=\varinjlim_{E\in\widehat\Vect(X)}\Vect(X,f^*E).
\end{equation*}
Elements of the direct limit above are, by definition, equivalence classes of triples $(E,V,i)\in\Tri(X,Y)$.
Two triples $(E,V,i)$ and $(E',V',i')$ are equivalent in the above direct limit iff they become isomorphic after pushing both to a common $E\hookrightarrow E''\hookleftarrow E'$, which is the same as saying $(E,V,i)$ and $(E',V',i')$ are isomorphic after adding elements of $\Vect(X)$.

Recall that a functor $F$ between filtered categories is called \emph{cofinal} iff (1) for every $x$ in the target there exists a morphism $x\to F(y)$, and (2) for every two morphisms $x\righttwoarrows F(y)$, there exists a morphism $y\to z$ such that the compositions $x\to F(z)$ coincide.
The significance of cofinality is that pulling back under a cofinal functor induces an isomorphism on direct limits:
\begin{equation*}
\varinjlim_{x\in\C}A(F(x))\xrightarrow\sim\varinjlim_{x\in\D}A(x)
\end{equation*}
is an isomorphism for any cofinal functor $F:\C\to\D$ between filtered categories $\C$ and $\D$ and any directed system of sets $A$ over $\D$.

The conclusion of Corollary \ref{enoughvectorbundles} for $f:Y\to X$ is equivalent to the assertion that the pullback functor $f^*:\widehat\Vect(X)\to\widehat\Vect(Y)$ is cofinal.
Indeed, $f^*$ always satisfies condition (2) since any two inclusions $E\hookrightarrow f^*F$ become homotopic upon postcomposing with $f^*(F\hookrightarrow F\oplus F)$, and condition (1) is precisely the conclusion of Corollary \ref{enoughvectorbundles}.

\begin{proposition}[Excision]\label{excision}
Let $X'\to X$ be a representable map of orbispaces satisfying the hypothesis of Theorem \ref{maingeneral}.
Let $Y\to X$ be arbitrary, and let $Y'\to X'$ denote its pullback along $X'\to X$.
If there is an open cover $X=U\cup V$ such that $Y\to X$ is an isomorphism over $U$ and $X'\to X$ is an isomorphism over $V$, then the natural map $K^0(X,Y)\to K^0(X',Y')$ is an isomorphism.
\end{proposition}

\begin{proof}
We may write the map in question in terms of direct limits as
\begin{equation*}
\varinjlim_{E\in\widehat\Vect(X)}\Vect(X,f^*E)\to\varinjlim_{E'\in\widehat\Vect(X')}\Vect(X',(f')^*E').
\end{equation*}
By Corollary \ref{enoughvectorbundles}, we may replace the second direct limit with the corresponding direct limit over $\widehat\Vect(X)$.
It thus suffices to show that the pullback map $\Vect(X,f^*E)\to\Vect(X',(f')^*E')$ is an isomorphism for every vector bundle $E$ on $X$ of bounded rank with $E':=E\times_XX'$.
Since $Y\to X$ is an isomorphism over $U$ and $X=U\cup V$, this coincides with the map $\Vect(V,f^*E|_V)\to\Vect(V\times_XX',(f')^*E|_V)$, which is a bijection since $X'\to X$ is an isomorphism over $V$.
\end{proof}

Given the significance of the hypothesis of Theorem \ref{maingeneral}, it is essential to show that this property is preserved under various natural operations.

\begin{lemma}\label{productcoarse}
For any topological stack $X$ and any locally compact Hausdorff space $R$, the natural map $\left|X\times R\right|\to\left|X\right|\times R$ is an isomorphism.
\end{lemma}

\begin{proof}
For any atlas $U\to X$, the induced map $U\to\left|X\right|$ is a topological quotient map, meaning it is surjective and a subset of the target is open iff its inverse image in the source is open.
Topological quotient maps are preserved by taking product with a locally compact Hausdorff space \cite[Lemma 4]{whiteheadquotient}, so $U\times R\to\left|X\right|\times R$ is also a topological quotient map.
Since $U\times R\to X\times R$ is also an atlas, the map $U\times R\to\left|X\times R\right|$ is also a topological quotient map.
Now $\left|X\times R\right|\to\left|X\right|\times R$ is a bijection, so we are done.
\end{proof}

\begin{lemma}\label{productfd}
If a topological space $X$ is coarsely finite-dimensional, then so is $X\times R$ for any closed subset $R\subseteq\RR^n$.
\end{lemma}

\begin{proof}
We begin with the case $R=[0,1]$.
Let a cover $X\times[0,1]=\bigcup_\alpha U_\alpha$ be given.
For every $x\in X$ and $t\in[0,1]$, there exists a pair of open neighborhoods $x\in V\subseteq X$ and $t\in W\subseteq[0,1]$ such that $V\times W$ is contained in some $U_\alpha$.
Fixing $x\in X$ and using compactness of $[0,1]$, we see that there are finitely many such $V_i\times W_i$ covering $\{x\}\times[0,1]$.
Hence there exists an open neighborhood $x\in V\subseteq X$ and an $\varepsilon>0$ such that $V\times((t-\varepsilon,t+\varepsilon)\cap[0,1])$ is contained in some $U_\alpha$ for every $t\in[0,1]$.
We consider the collection of all such pairs $(V,\varepsilon)$.
Since $X$ is coarsely finite-dimensional, there is a collection of such pairs $(V_\alpha,\varepsilon_\alpha)$ such that the nerve of the covering $X=\bigcup_\alpha V_\alpha$ is finite-dimensional.
Now cover $X\times[0,1]$ by $V_\alpha\times([0,1]\cap\varepsilon_\alpha\cdot(k,k+2))$ for integers $k$.

Next, we consider the case $R=\RR$.
Let a cover of $X\times\RR$ be given.
The intersection of this cover with $X\times[n,n+2]$ has a refinement with finite-dimensional nerve by the case $R=[0,1]$.
Consider this refinement intersected with $X\times(n,n+2)$, and take union over all integers $n$ to obtain a refinement of the original cover of $X\times\RR$ with finite-dimensional nerve.

Finally, by induction we obtain the case $R=\RR^n$, and the general case follows since coarse finite-dimensionality passes to closed subsets.
\end{proof}

\begin{corollary}\label{hypstable}
If $X$ satisfies the hypothesis of Theorem \ref{maingeneral}, then so does $X\times R$ for any closed subset $R\subseteq\RR^n$.
\end{corollary}

Let us now recall the Puppe sequence, which produces from Propositions \ref{lesktheory} and \ref{excision} a long exact sequence.
Let $(X,A)$ be any pair of orbispaces satisfying the hypothesis of Theorem \ref{maingeneral}.
Let $I:=[0,1]$.
We now have maps of pairs
\begin{align*}
&(A,\varnothing)\to(X,\varnothing)\to(X,A)\\
&(A\times I,A\times\partial I)\to(X\times I,X\times\partial I)\to(X\times I,(X\times\partial I)\cup(A\times I))\\
&(A\times I^2,A\times\partial I^2)\to(X\times I^2,X\times\partial I^2)\to(X\times I^2,(X\times\partial I^2)\cup(A\times I^2))\\
&\qquad\vdots
\end{align*}
(note that these pairs satisfy the hypothesis of Theorem \ref{maingeneral} by Corollary \ref{hypstable}).
We also have maps $(X\times I^k,(X\times\partial I^k)\cup(A\times I^k))\to(A\times I^{k+1},A\times\partial I^{k+1})$ up to inverting maps inducing isomorphisms on $K^0$.
Namely, these `connecting maps' are given by
\begin{align*}
&(X\times I^k\times I,(X\times\partial I^k\times I)\cup(A\times I^k\times\{1\}))\\
&\qquad\qquad\downarrow\\
&(X\times I^k\times I,(X\times\partial I^k\times I)\cup(A\times I^k\times\{1\})\cup(X\times I^k\times\{0\}))
\end{align*}
where we note that the domain admits a natural map to $(X\times I^k,(X\times\partial I^k)\cup(A\times I^k))$ (which induces an isomorphism on $K^0$ by Proposition \ref{relhtpyinv}) and the target a natural map from $(A\times I^{k+1},A\times\partial I^{k+1})$.
To see that this second map also induces an isomorphism on $K^0$, factor it as
\begin{align*}
(A\times I^k\times I&,(A\times\partial I^k\times I)\cup(A\times I^k\times\{1\})\cup(A\times I^k\times\{0\}))\\
&\qquad\downarrow\\
((X\times I^k\times\{0\})\cup(A\times I^k\times I)&,(A\times\partial I^k\times I)\cup(A\times I^k\times\{1\})\cup(X\times I^k\times\{0\}))\\
&\qquad\downarrow\\
((X\times I^k\times I)&,(X\times\partial I^k\times I)\cup(A\times I^k\times\{1\})\cup(X\times I^k\times\{0\}))
\end{align*}
and observe that the first map is an isomorphism by excision Proposition \ref{excision} (and some manipulation involving adding $A\times I^k\times[0,\frac 12]$ to the second space of both pairs, which does not change their homotopy type but makes it possibly to apply Proposition \ref{excision})
and the second map is an isomorphism by Proposition \ref{relhtpyinv} (a homotopy inverse to $(X\times I^k\times\{0\})\cup(A\times I^k\times I)\to X\times I^k\times I$ is given by projection to $X\times I^k\times\{0\}$, and the second terms are both homotopy equivalent to $(X\times D^k)\cup(A\times S^k)$).

\begin{definition}
For any pair $(X,A)$, we define\footnote{To keep track of signs, one should really write this as $K^{-n}(X)\otimes H^n(I^n,\partial I^n)=K^0(X\times I^n,X\times\partial I^n)$ etc., however for the purposes of our presentation here, we will not be so precise.}
\begin{align*}
K^{-n}(X):&=K^0(X\times I^n,X\times\partial I^n),\\
K^{-n}(X,A):&=K^0(X\times I^n,(X\times\partial I^n)\cup(A\times I^n)).
\end{align*}
\end{definition}

Given this definition, we can write the Puppe sequence above in a more familiar form.
Namely, for $X$ an orbispace satisfying the hypothesis of Theorem \ref{maingeneral}, we have a sequence of the form
\begin{equation*}
\cdots\to K^{-2}(A)\to K^{-1}(X,A)\to K^{-1}(X)\to K^{-1}(A)\to K^0(X,A)\to K^0(X)\to K^0(A)
\end{equation*}
functorial in the pair $(X,A)$.

\begin{proposition}[Long Exact Sequence]
The sequence above is exact.
\end{proposition}

\begin{proof}
Each of the following three triples satisfies the hypotheses of Proposition \ref{lesktheory}:
\begin{align*}
&(A\times I^k,A\times\partial I^k)\to(X\times I^k,X\times\partial I^k)\to(X\times I^k,(X\times\partial I^k)\cup(A\times I^k))\\
&(X\times I^k,X\times\partial I^k)\xrightarrow{\times\{0\}}(X\times I^k\times I,(X\times\partial I^k\times I)\cup(A\times I^k\times\{1\}))\\
&\qquad\qquad\qquad\qquad\quad\to(X\times I^k\times I,(X\times\partial I^k\times I)\cup(A\times I^k\times\{1\})\cup(X\times I^k\times\{0\}))\\
&(X\times I^k\times[{\textstyle\frac 12},1],(X\times\partial I^k\times[{\textstyle\frac 12},1])\cup(A\times I^k\times\{1\}))\\
&\qquad\qquad\to(X\times I^k\times I,(X\times\partial I^k\times I)\cup(X\times I^k\times\{0\})\cup(A\times I^k\times\{1\}))\\
&\qquad\qquad\to(X\times I^k\times I,(X\times\partial I^k\times I)\cup(X\times I^k\times\{0\})\cup(X\times I^k\times[{\textstyle\frac 12},1]))
\end{align*}
and the last pair is homotopy equivalent to $(X\times I^k\times[0,\frac 12],(X\times\partial I^k\times[0,\frac 12])\cup(X\times I^k\times\{0\})\cup(X\times I^k\times\{\frac 12\}))$
\end{proof}

We now define a multiplicative structure on $K$-theory.
There is a natural map $K^0(X)\otimes K^0(Y)\to K^0(X\times Y)$ sending $[E]\otimes[F]$ to $[E\boxtimes F]$.
To define the tensor product map on relative $K$-theory, we will consider the following alternative model based on chain complexes.

\begin{definition}
Let $(X,A)$ be a pair.
Consider bounded complexes of vector bundles on $X$ which are exact over $A$.
Let us call two such complexes homotopic iff they are the restrictions to $\times\{0\}$ and $\times\{1\}$ of a complex on $X\times[0,1]$ acyclic over $A\times[0,1]$ (homotopy is an equivalence relation).
Denote the set of homotopy classes of such complexes by $\Kom(X,A)$, which is an abelian monoid under direct sum, and define
\begin{equation*}
K^0_\kom(X,A):=\Kom(X,A)/\Kom(X,X).
\end{equation*}
This quotient is an abelian group: indeed, the sum of a complex and its shift is homotopic to the mapping cone of the identity map.
\end{definition}

\begin{proposition}
There is a natural isomorphism $K^0(X,A)=K^0_\kom(X,A)$ for paracompact orbispace pairs $(X,A)$.
\end{proposition}

\begin{proof}
There is a natural map $K^0(X,A)\to K^0_\kom(X,A)$ defined as follows.
Given a triple $(E_0,E_1,i)\in\Tri(X,A)$, extend the isomorphism $i$ to a neighborhood of $A$ using Lemma \ref{sectionextension}, and multiply by a cutoff function supported inside this neighborhood to obtain a globally defined map $d:E_0\to E_1$ which is an isomorphism over $A$.
This complex is well defined up to homotopy, so we have defined a map $\Tri(X,A)\to\Kom(X,A)$.
This map evidently sends $\Vect(X)$ to (the image of) $\Kom(X,X)$, so it defines a map $K^0(X,A)\to K^0_\kom(X,A)$ as desired.

We may also define a natural map $K^0_\kom(X,A)\to K^0(X,A)$ as follows.
Given a complex $(E_\bullet,d)$, choose a metric on each $E_i$ using Lemma \ref{paracompactmetric}, and note that the map $d+d^*:E_\mathrm{even}:=\bigoplus_{i\text{ even}}E_i\to\bigoplus_{i\text{ odd}}E_i=:E_\mathrm{odd}$ is an isomorphism wherever $d$ is exact.
Since all metrics are homotopic and $K^0(X,A)$ is homotopy invariant by Proposition \ref{relhtpyinv}, sending $(E_\bullet,d)$ to the triple $(E_\mathrm{even},E_\mathrm{odd},(d+d^*)|_A)$ gives a well defined map $K^0_\kom(X,A)\to K^0(X,A)$.
The composition $K^0(X,A)\to K^0_\kom(X,A)\to K^0(X,A)$ is evidently the identity.

To finish, it suffices to show that $K^0(X,A)\to K^0_\kom(X,A)$ is surjective.
Equivalently, we are to show that every element of $K^0_\kom(X,A)$ is represented by a complex concentrated in degrees $[0\;\;\;1]$.
To do this, we use the following ``folding'' operation.
Given a complex concentrated in degrees $[s\;\;\;r]$, the differential $d_r:E_r\to E_{r-1}$ is injective over $A$, hence over an open neighborhood $U$ of $A$ inside $X$.
If $d_r$ is in fact everywhere injective, then we may choose a metric on $E_{r-1}$ and split it as $\im(E_r)\oplus\im(E_r)^\perp$, showing that the subcomplex $E_r\to\im(E_r)$ is a direct summand, so our complex represents the same class in $K^0_\kom$ as one concentrated in degrees $[s\;\;\;r-1]$.
Now consider the general case where $d_r$ is not assumed everywhere injective.
By replacing our given complex by its direct sum with the mapping cone of the identity map $E_r\to E_r$ concentrated in degrees $[r-2\;\;\;r-1]$, we may assume that there exists an injection $i:E_r\to E_{r-1}$ which agrees with $d_r$ over a neighborhood $U$ of $A$.
Now multiply $d$ by a cutoff function $\varphi$ supported inside $U$ and positive on $A$ (this is homotopic to the original differential), and take the convex interpolation between $\varphi d_r$ and $i$ as a further homotopy.
This ensures that the differential $E_r\to E_{r-1}$ is everywhere injective, so our complex represents the same class in $K^0_\kom$ as one concentrated in degrees $[\min(s,r-2)\;\;\;r-1]$.
The dual operation shows that any complex concentrated in degrees $[s\;\;\;r]$ represents the same class in $K^0_\kom$ as one concentrated in degrees $[s+1\;\;\;\max(s+2,r)]$.
Combining these two operations, we can put anything in degrees $[0\;\;\;1]$.
\end{proof}

Tensor product at the level of $K^0_\kom$ evidently defines commutative and associative maps
\begin{equation*}
K^0(X,A)\otimes K^0(Y,B)\to K^0(X\times Y,(A\times Y)\cup(X\times B))
\end{equation*}
for paracompact orbispace pairs $(X,A)$ and $(Y,B)$ for which $X\times Y$ is paracompact.
These induce, by inspection, associative and graded commutative maps
\begin{equation*}
K^{-n}(X,A)\otimes K^{-m}(Y,B)\to K^{-n-m}(X\times Y,(A\times Y)\cup(X\times B))
\end{equation*}
under the same hypotheses (note that a paracompact space times a compact space is paracompact).

We conclude with a discussion of Bott periodicity.
So far, our discussion has applied equally to complex vector bundles as to real vector bundles, however we now restrict to complex vector bundles.
The \emph{Bott element} $\beta\in K^{-2}(*)=K^0(D^2,\partial D^2)$ is represented by the complex of vector bundles $\underline\CC\xrightarrow{\cdot z}\underline\CC$ on $\CC$ (which contains $D^2$ as the unit disk).

\begin{proposition}
For any finite simplicial complex of groups $X$ with subcomplex $A$, multiplication by the Bott element $K^{-k}(X,A)\to K^{-k-2}(X,A)$ is an isomorphism for all $k\geq 0$.
\end{proposition}

\begin{proof}
By the long exact sequence, excision, the five lemma, and finite additivity, we are reduced to the case of $(D^i,\partial D^i)\times\BB G$.
By the definition of $K^{-n}$, we are further reduced to the case of $K^0$, namely to showing that multiplication by the Bott element
\begin{equation*}
K^0((D^k,\partial D^k)\times\BB G)\to K^0((D^{k+2},\partial D^{k+2})\times\BB G)
\end{equation*}
is an isomorphism for all $k\geq 0$, which is well known \cite{atiyahbott} \cite[Theorem 4.3]{atiyah} \cite[\S 3]{segalequivariant}.
\end{proof}

In fact, the argument above shows more generally that for any complex line bundle $L$ over a finite simplicial complex of groups $X$, pullback followed by multiplication with the relative Bott class in $K^0(L,\infty)$ (represented by the complex $\underline\CC\xrightarrow{\cdot p}L$ on $L$) defines an isomorphism $K^{-k}(X)\to K^{-k}(L,\infty)$ (and, yet more generally, that pullback and multiplication with the Koszul complex defines an isomorphism $K^{-k}(X)\to K^{-k}(E,\infty)$ for any complex vector bundle $E$).

\begin{definition}
The periodic $K$-theory of a pair $(X,A)$ is  the direct limit
\begin{equation*}
K^*_\per(X,A):=\varinjlim_jK^{*-2j}(X,A)
\end{equation*}
over multiplication by $\beta$.
Whereas $K^*$ is defined only in non-positive degrees, $K^*_\per$ is defined in all degrees.
\end{definition}

Since direct limits are exact, $K^*_\per$ is, like $K^*$, a cohomology theory for orbispace pairs satisfying the hypothesis of Theorem \ref{maingeneral}.
When the natural map $K^*\to K^*_\per$ is an isomorphism, $K^*_\per$ provides a natural extension of $K^*$ from non-positive degrees to all degrees.

\begin{corollary}
For any finite simplicial complex of groups $X$ with subcomplex $A$, the natural map $K^*(X,A)\to K^*_\per(X,A)$ is an isomorphism in non-positive degrees.
\qed
\end{corollary}

\bibliographystyle{amsplain}
\bibliography{orbibundle}
\addcontentsline{toc}{section}{References}

\end{document}